\documentclass[11pt, a4paper, UKenglish]{article}

\usepackage{amsfonts}
\usepackage{amsmath}
\usepackage{amssymb}
\usepackage{amsthm}
\usepackage{array}
\usepackage[UKenglish]{babel}
\usepackage{circuitikz}
\usepackage{color}
\usepackage{enumerate}
\usepackage{graphicx}
\usepackage{hyperref}
\usepackage{latexsym}
\usepackage{makecell}
\usepackage{makeidx}
\usepackage{pgfplots}
\usepackage{tikz}
\usepackage{tkz-fct}
\usepackage{tkz-graph}
\usepackage{upgreek}

\renewcommand{\phi}{\varphi}
\renewcommand{\tau}{\uptau}

\theoremstyle{plain}
\newtheorem{theorem}{Theorem}[subsection]
\newtheorem{cor}[theorem]{Corollary}
\newtheorem{lemma}[theorem]{Lemma}

\theoremstyle{definition}
\newtheorem{definition}[theorem]{Definition}
\theoremstyle{remark}

\usepackage{caption}
\usepackage{subcaption}
\usepackage{fancyhdr}
\usepackage{setspace}
\date{\today}

\newcommand{\Grun}{Gr\"un~}

\begin{document}
\thispagestyle{empty}
\begin{center}
\begin{spacing}{1.8} 
{\Huge On the characterization of the numbers 
$n$ such that any group of order $n$ has a given property $P$}\\
\end{spacing}
\vspace*{0.7cm}
{\huge Logan Crew}\\
\vspace*{1cm}
{\huge Advisor: Professor\ Thomas Haines}\\
\vfill
\textsc{\LARGE Honors Thesis in Mathematics}\\
\vspace*{0.3cm}
\textsc{\LARGE University of Maryland, College Park}\\
\end{center}
\pagebreak
\thispagestyle{empty}
\tableofcontents
\pagebreak
\section{Introduction}

\markboth{Logan Crew}
{On the characterization of $P$-groups}
One of the classical problems in group theory is 
determining the set of positive integers $n$ such that 
every group of order $n$ has a particular property $P$, such as cyclic or abelian.  
We first present the Sylow theorems and the idea of solvable groups, 
both of which will be invaluable in our analysis. 
We then gather various solutions to this problem for 
cyclic, abelian, nilpotent, and supersolvable groups, 
as well as groups with ordered Sylow towers.  
There is also quite a bit of research on this problem for solvable
groups, but this research is outside the scope of this paper.

This work is an exposition of known results, but it is hoped that the reader will find 
useful the presentation in a single account of the various tools 
that have been used to solve this general problem.  This article claims no
originality, but is meant as a synthesis of related knowledge and resources.    

To simplify terminology, if a positive integer $n$ satisfies that every group of order $n$ has property $P$, we will call
$n$ a \emph{P number}.  For example, if the only group of order $n$ is the cyclic group, we will call $n$ a \emph{cyclic number}, 
and similarly for \emph{abelian number}, \emph{nilpotent number}, \emph{ordered Sylow number}, and
\emph{supersolvable number}.

Some notation:
\begin{itemize}
\item $C_k$ denotes the cyclic group of order $k$.
\item $(a,b)$ denotes the gcd of $a$ and $b$.
\item $Z(G)$ denotes the center of the group $G$.
\item $G'$ denotes the commutator subgroup of $G$.
\item $|G : H|$ denotes the index of a subgroup $H$ in a group $G$.
\end{itemize}

\section{Background Information}

Before continuing to the main results, we establish some foundations.  

\subsection{Sylow's Theorem}

The results here are adapted from \cite[p. 139]{dummit}.

\begin{definition}

Let $G$ be a finite group, and let $p$ be a prime dividing its order.  
If the order of $G$ may be written as $p^am$ where $p \nmid m$, then a subgroup of $G$ with order
$p^a$ is called a \emph{Sylow $p$-subgroup} of $G$, and is usually denoted by $P_p$.
The number of Sylow $p$-subgroups of $G$ will be denoted by $n_p$.

\end{definition}

Now for Sylow's Theorem:

\begin{theorem} \label{thm:syl}

Let $G$ be a group of order $p^am$, where $p \nmid m$.
\begin{enumerate}

\item There is at least one Sylow $p$-subgroup of $G$
\item If $P$ is any Sylow $p$-subgroup of $G$ and $Q$ is any $p$-subgroup of $G$, there exists
$g \in G$ such that $Q \subseteq gPg^{-1}$, that is, every $p$-subgroup of $G$ is 
contained in some conjugate of any Sylow $p$-subgroup, and in particular 
any two Sylow $p$-subgroups are conjugate
\item The number $n_p$ of Sylow $p$-subgroups in $G$ satisfies $n_p \equiv 1$ (mod p)
and furthermore, $n_p$ is equal to $|G : N_G(P)|$, the index of the normalizer of any 
Sylow $p$-subgroup in $G$

\end{enumerate}

\end{theorem}

First, we prove an auxiliary lemma about Sylow subgroups:

\begin{lemma} \label{lem:syl}

If $P$ is a Sylow $p$-subgroup of $G$, and $Q$ is any $p$-subgroup of $G$, then $Q \cap N_G(P) = Q \cap P$.

\end{lemma}

\begin{proof}

Write $Q \cap N_G(P) = H$.  Since $P \subseteq N_G(P)$, certainly $Q \cap P \subseteq Q \cap N_G(P)$, so we must only
prove the reverse inclusion.  Since clearly $H \subseteq Q$, we must prove that $H \subseteq P$.  

Since $H \subseteq N_G(P)$, we know that $PH$ is a subgroup of $G$.  We can also see that $PH$ is a $p$-group
that contains $P$ as a subgroup, so we can only have that $PH = P$, so indeed $H \subseteq P$.

\end{proof}

Now we are ready to prove Sylow's theorem in full:

\begin{proof}

We first prove the existence of Sylow subgroups of $G$ by induction on its order.  The base case is trivial, so assume that
for a given group $G$, all groups of order less than $G$ have Sylow subgroups.
For a prime $p$ dividing the order of $G$, we first suppose that $p \mid ord(Z(G))$.  In this case, Cauchy's theorem
implies that $Z(G)$ has a subgroup $N$ of order $p$.  Consider the group $G/N$, which has order $p^{a-1}m$.
The induction hypothesis implies that $G/N$ has a Sylow $p$-subgroup $P/N$ of order $p^{a-1}$.  But then $P$ is a subgroup
of $G$ having order $p^a$, so $G$ also has a Sylow $p$-subgroup.

Now, suppose that $p \nmid ord(Z(G))$.  Let $g_1, g_2, \dots g_r$ be representatives of the distinct non-central conjugacy classes
of $G$, and apply the class equation to get
$$ ord(G) = ord(Z(G)) + \sum\limits_{i=1}^r |G : C_G(g_i)| $$
Since $p \mid ord(G)$ but $p \nmid ord(Z(G))$ by assumption, there is some $i$ such that
$p \nmid |G : C_G(g_i)|$.  For this $i$, $ord(C_G(g_i)) = p^ak$, with $k < m$ since
$g_i \notin Z(G)$.  By the inductive hypothesis, this $C_G(g_i)$ has a subgroup of order 
$p^a$, and therefore so does $G$.  This proves $1$.

Now, for a given prime $p$ dividing the order of a group $G$, we know that there exists
some Sylow $p$-subgroup of $G$, which we denote by $P$.  Let 
$S = \{P_1,P_2, \dots ,P_r\}$ be the set of all conjugates of $P$ in $G$, and let $Q$ be any
$p$-subgroup of $G$.  The conjugation action of $Q$ on $S$ produces orbits
$O_1 \cup O_2 \cup \dots \cup O_s = S$.  Note that $r = \sum\limits_{i=1}^s ord(O_i)$.

Without loss of generality, relabel the elements of $S$ so that the first $s$ conjugates
are representatives of the orbits, so $P_i \in O_i$ for $1 \leq i \leq s$.   
For each $i$, we have that $ord(O_i)ord(N_Q(P_i)) = ord(Q)$.  However,
$N_Q(P_i) = N_G(P_i) \cap Q = P_i \cap Q$ by Lemma \ref{lem:syl}, so 
$ord(O_i) = |Q : P_i \cap Q|$ for $1 \leq i \leq s$.  

Now, note that $r$ is the number of conjugates of $P$ in $G$, and is therefore independent of
our choice of group $Q$ in the above conjugation action.  Therefore, we may calculate $r$ by taking
the specific case $Q = P_1$.  In this case certainly $ord(O_1) = 1$, while for $i > 1$, we must have
$ord(O_i) = |P_1 : P_1 \cap P_i| > 1$.  Furthermore, $|P_1 : P_1 \cap P_i|$ must be divisible by $p$,
since every group in the expression is a $p$-group.  
Since $r = ord(O_1) + \sum\limits_{i = 2}^s ord(O_i)$, we must have that 
$r \equiv 1 \text{ (mod }p)$.

Let $Q$ be any $p$-subgroup of $G$, and suppose that $Q$ is contained in no
conjugate of $P$.  Then certainly $Q \cap P_i \subsetneq Q$ for all $1 \leq i \leq s$,
so $ord(O_i) = |Q : Q \cap P_i|$ is divisible by $p$ for all $1 \leq i \leq s$, contradicting
$r \equiv 1 \text{ (mod }p)$.  Thus, $Q$ is contained in some conjugate of every Sylow $p$-subgroup of $G$,
which also implies that every Sylow $p$-subgroup of $G$ is conjugate.  This proves $2$.

But now every Sylow $p$-subgroup of $G$ is conjugate, and so there are precisely the $r$
that were in our set $S$, so $r = n_p$, and $n_p \equiv 1 \text{ (mod }p)$.  Furthermore, if all
Sylow $p$-subgroups are conjugate, we must also have $n_p = |G : N_G(P)|$ for any
Sylow $p$-subgroup of $G$, and this proves $3$.   

\end{proof}

The following characteristics of Sylow subgroups are also important:

\begin{cor} \label{cor:syl}

Let $P$ be a Sylow $p$-subgroup of $G$.
The following statements are equivalent:

\begin{enumerate}
\item $P$ is the unique Sylow $p$-subgroup of $G$
\item $P$ is normal in $G$
\item $P$ is characteristic in $G$
\end{enumerate}

\end{cor}

\begin{proof}

Suppose $P$ is the unique Sylow $p$-subgroup of $G$.  Then since the set of conjugates
of $P$ in $G$ is precisely the set of Sylow $p$-subgroups in $G$, we can only have that
$gPg^{-1} = P$ for all $g \in G$, so $P \triangleleft G$.  Conversely, if $P \triangleleft G$,
then for any other Sylow $p$-subgroup $Q$ of $G$, we have $Q$ is conjugate to $P$, so for some 
$g \in G$ we have $Q = gPg^{-1} = P$, so $P$ is the unique Sylow $p$-subgroup of $G$,
and $1$ and $2$ are equivalent.  However, characteristic implies normal, and uniqueness implies 
characteristic, so all of $1$, $2$, and $3$ are equivalent.

\end{proof}

\subsection{Solvable Groups}

The results in this section are adapted from \cite[p. 138]{hall}.

\begin{definition}

The \emph{derived series} of a group $G$ is the series
$$ G = G_0 \supseteq G_1 \supseteq G_2 \supseteq \dots $$
where $G_{i+1}$ is the commutator subgroup of $G_i$.
A group $G$ is said to be \emph{solvable} if its derived series terminates in the identity
after a finite number of terms.

\end{definition}

For example, $S_3$ has derived series $S_3, A_3, (1)$, so $S_3$ is solvable.  The smallest example of a nonsolvable group is
$A_5$, and in general $A_k$ is not solvable for any $k \geq 5$.  

It is not hard to show that if $G$ is solvable and $H$ is normal in $G$ that $H$ and $G/H$ are solvable.  Of particular
interest is that the converse is true:

\begin{lemma} \label{lem:solv}

Suppose that $G$ is a group with a normal subgroup $H$, and suppose further that both $H$ and 
$G/H$ are solvable.  Then $G$ is solvable.

\end{lemma}

\begin{proof}

Since $G/H$ is solvable, its derived series is of the form 
$$G/H = G_0/H \supseteq G_1/H \supseteq \dots \supseteq \{e\} = H/H$$
On the other hand, this quotient series lifts to the first part of the derived series for 
$G$ satisfying 
$$G = G_0 \supseteq G_1 \supseteq \dots \supseteq H$$
but since $H$ is solvable, the remainder of the derived series for $G$ terminates at the identity,
so $G$ is solvable.

\end{proof}

\subsection{Hall $\pi$-subgroups}

The results in this section are taken from \cite[p.200]{dummit}

\begin{definition}

Let $G$ be a finite group with order $n$, and let $\pi$ be a set of primes dividing $n$.  A \emph{Hall $\pi$-subgroup} $H$
of $G$ is a subgroup such that every prime dividing $|H|$ is in $\pi$, and no prime dividing $|G:H|$ is in $\pi$.      

\end{definition}

As an example, the alternating group of five elements $A_5$ has order $60 = 2^2 \times 3 \times 5$, so
$A_4$ is a Hall $\{2,3\}$-subgroup of $A_5$.

While all possible Sylow $p$-subgroups always exist for any finite group $G$, Hall $\pi$-subgroups need not exist for all
possible $\pi$.  For example, while $A_5$ has a Hall $\{2,3\}$-subgroup, it has neither a
Hall $\{2,5\}$-subgroup nor a Hall $\{3,5\}$-subgroup.  However, if $G$ is solvable, then it has all possible Hall
subgroups: 
   
\begin{theorem} \label{thm:hall}

If $G$ is any finite solvable group and $\pi$ is any set of primes, then G has a Hall-$\pi$
subgroup.

\end{theorem}

We first prove a very important lemma that will be used here and again in the section on
supersolvable numbers:

\begin{lemma} \label{lem:elem}

Let $G$ be a finite solvable group, and let $N$ be a minimal normal subgroup of $G$ (meaning that
$N \triangleleft G$, and when there exists $H \triangleleft G$ 
such that $\{e\} \subseteq H \subseteq N$, it must
be that $H = \{e\}$ or $H = N$).  Then $N$ must be elementary abelian 
(i.e. is an abelian group where every nonidentity element has the same prime order).

\end{lemma}

\begin{proof}

Note that $N$ is solvable since it is a subgroup of
the solvable group $G$.  If $N = \{e\}$, the claim is obvious.  
If $N$ is nontrivial, we consider its commutator subgroup 
$[N,N]$.  This commutator is characteristic in $N$, so since $N \triangleleft G$, we have
$[N,N] \triangleleft G$.  Since $N$ is minimal normal, it must be that $[N,N] = \{e\}$ or
$[N,N] = N$.  However, $[N,N] = N$ would imply that the derived series for $N$ is a single repeating term
that is not the identity, so then $N$ would not be solvable, a contradiction.  It must be then that
$[N,N] = \{e\}$, which implies that $N$ is abelian.

Now let $p$ be a prime that divides $|N|$.  Then $N$ has Sylow $p$-subgroups which are all conjugate to
each other, but since $N$ is abelian, conjugation is trivial, so there is a unique 
Sylow $p$-subgroup $P$.  Then $P$ is characteristic in $N$ and $N \triangleleft G$, so 
$P \triangleleft G$.  Clearly $P$ is nontrivial, so $P = N$, and $N$ is an
abelian $p$-group.

Finally, we let $pN$ denote the set $n^p | n \in N$.  It is easy to check that this is a subgroup of $N$, and that
it is characteristic since it is invariant under isomorphisms of $N$.  Again, we may conclude that
$pN = \{e\}$ or $pN = N$.  However, Cauchy's theorem tells us that $N$ has at least one element of order $p$, so
that $pN$ is strictly contained in $N$.  Therefore, $pN = \{e\}$, so every element of $N$ must have order $p$,
implying that $N$ is elementary abelian.       
  
\end{proof}

Now for the existence of Hall $\pi$-subgroups:

\begin{proof}

Fix any set $\pi$ of primes.  We prove by induction that every finite solvable group $G$ has 
Hall $\pi$-subgroups.  In the base case, the trivial group is a Hall $\pi$-subgroup of itself for
any $\pi$, and certainly our particular choice.  Now let $G$ be a finite solvable group of order $n$, 
and suppose that every solvable group of order less than $n$ has a Hall $\pi$-subgroup.  Let $N$ be some 
minimal normal subgroup of $G$.  By the above lemma, $N$ is elementary abelian, and in particular $N$
is a $p$-group for some prime $p$.  Since quotient groups of solvable groups are solvable, 
it must be that $G/N$ is solvable, so by the inductive hypothesis, $G/N$ has a Hall
$\pi$-subgroup.  If $p \in \pi$, then the Hall $\pi$-subgroup of $G/N$ naturally corresponds to 
a Hall $\pi$-subgroup of $G$.

Now we suppose that $p$ is not in $\pi$.  It is still true that $G/N$ has a Hall $\pi$-subgroup
$K/N$, and this corresponds to a subgroup $K$ of $G$.  Since subgroups of solvable groups
are solvable, if $K$ is a proper subgroup of $G$, then by the inductive hypothesis $K$ has a 
Hall $\pi$-subgroup $H$.  Now every prime dividing $|H|$ is in $\pi$, but no prime dividing $|G:K|$ or 
$|K : H|$ is in $\pi$, so no prime dividing $|G:K||K:H| = |G:H|$ is in $\pi$.  It follows that $H$ is a 
Hall $\pi$-subgroup of $G$.  

It remains to consider when $p$ is not in $\pi$, and the Hall $\pi$-subgroup of $G/N$ is itself.  In this case,
since $N$ is a $p$-group and $|G:N|$ is relatively prime to $|N|$, it must be that 
$N$ is a Sylow $p$-subgroup of $G$.  Let $M/N$ be a minimal normal subgroup of $G/N$, which also
corresponds to a normal subgroup $M$ of $G$.  It must be that $M/N$
is of prime power order, but it cannot be of $p$-power order, since $p \nmid |G : N|$.  Thus,
$M/N$ is an elementary abelian group of $q$-power order for some prime $q \neq p$.  
Note that $p$ and $q$ are then
the only distinct prime factors dividing $M$.  Also, $q$ divides $|G:N|$, and since $G/N$ is the Hall $\pi$-subgroup of itself, it must be that
$q \in \pi$.  Since $q$ divides $|M|$, $M$ has some Sylow $q$-subgroup $Q$.  If $Q \triangleleft G$, then 
$G/Q$ has a Hall $\pi$-subgroup $H/Q$, but since $Q$ is of $q$-power order and $q \in \pi$, the corresponding
subgroup $H$ of $G$ is a Hall $\pi$-subgroup.

We now prove Frattini's Argument, which states that if $M \triangleleft G$ and $Q$ is a Sylow $q$-subgroup of $M$ then
$G = MN_G(Q)$.  To show this, consider the subgroup $gQg^{-1}$ formed by conjugating $Q$ with any $g \in G$. 
Since $Q \subseteq M$ and $M \triangleleft G$, $gQg^{-1} \subseteq M$.   Since all Sylow $q$-subgroups of $M$ are 
conjugate, it must be that $Q$ and $gQg^{-1}$ are conjugate in $M$, so 
there must be some $m \in M$ such that $gQg^{-1} = mQm^{-1}$.  Rearranging, we have that
$m^{-1}gQg^{-1}m = Q$, so that $m^{-1}g \in N_G(Q)$.  But then any $g \in G$ may be written as
$m(m^{-1}g)$, a product of elements in $M$ and $N_G(Q)$, so $G = MN_G(Q)$.

Continuing the original proof, we may now assume that $Q$ is not normal in $G$.   By Frattini's argument,
$G = MN_G(Q)$, and it follows that $|G|\times|M \cap N_G(Q)| = |M| \times |N_G(Q)|$, or equivalently
$|G|/|N_G(Q)| = |M|/|M \cap N_G(Q)|$.  Note that since $Q \subseteq M \cap N_G(Q)$, we have that
$|M|/|M \cap N_G(Q)|$ divides $|M|/|Q|$.  However, the only prime factors of $|M|$ are $p$ and $q$, so since
$Q$ is a Sylow $q$-subgroup of $M$, $|M|/|Q|$ is a power of $p$, and then so is
$|M|/|M \cap N_G(Q)| = |G|/|N_G(Q)|$.  Since $Q$ is not normal in $G$, 
$N_G(Q)$ is a proper subgroup of $G$ that has a Hall $\pi$-subgroup $H$ by the inductive hypothesis.  Since
$|G|/|N_G(Q)|$ is a power of $p \notin \pi$, $H$ is also a Hall $\pi$-subgroup of $G$, completing the proof.   

\end{proof}

\section{Cyclic Numbers}

Lagrange's theorem tells us that any prime number $p$ must be a cyclic number, but there are also many composite cyclic numbers.
According to a well-known application of Cauchy's theorem, if $p$ and $q$ are prime numbers with $p>q$ and $q\nmid p-1$, 
then every group of order $pq$ is cyclic (\cite[p. 91]{herstein}).  This basic idea motivates the following theorem:

\begin{theorem}

A positive integer $n$ is a cyclic number $\iff (n, \phi(n)) = 1$.

\end{theorem}

The following proof is from \cite{haines1}:

\begin{proof}

Suppose that there is some prime $q$ such that $q^2\mid n$.  Then since $q\mid \phi(q^k)$ for $k>1$, we have $q\mid \phi(n)$ and $q\mid n$, so 
certainly $(n, \phi(n)) \neq 1$.  Furthermore, the group $C_q \times C_{n/q}$ is a group of order $n$ that cannot be cyclic, as $(q, n/q) \neq 1$.

For the remainder of the proof, we may assume that $n$ is square-free, so $n$ = $p_1p_2 \dots p_r$ 
for distinct primes $p_1, p_2, \dots, p_r$.  Suppose that $(n, \phi(n)) \neq 1$.  Then since $n$ is square-free, we must have that
$p_i\mid p_j-1$ for some primes $p_i$ and $p_j$ dividing $n$.  Therefore, there exists a nontrivial homomorphism 
$h: C_{p_i}\rightarrow Aut(C_{p_j})$, and a corresponding nontrivial semidirect product $C_{p_j} \rtimes_h C_{p_i}$ which produces a noncyclic group
of order $p_ip_j$.  It follows that the direct product of this group with $C_{n/p_ip_j}$ is a group of order $n$
which is not cyclic, proving that $(n, \phi(n)) = 1$ is necessary for $n$ to be a cyclic number.

To prove sufficiency, we proceed by contradiction.  Let $n$ be the smallest positive integer that is square-free and
satisfies $(n, \phi(n)) = 1$ but is not a cyclic number, and take $G$ to be a group of order $n$ that is not cyclic.  Cauchy's theorem tells us
that there exist elements $x_i$ in $G$ for $1\leq i \leq r$ such that $ord(x_i) = p_i$ for each $i$.  If $G$ were abelian,
the product $x = x_1x_2 \dots x_r$ would be an element of order $n$, so $G$ would be cyclic, a contradiction. It is also clear that 
if $d\mid n$, then $(n, \phi(d)) = 1$, and $(d, \phi(d)) = 1$.  In particular, this means that the minimality of $n$ as a 
supposed counterexample implies if $d\mid n$ and $d < n$, then $d$ is a cyclic number, 
so it follows that all proper subgroups and quotient groups of $G$ are cyclic.

We now prove that $Z(G)$ is trivial.  Suppose that $ord(Z(G)) > 1$.  
In this case, the quotient group $G/Z(G)$ must be cyclic as noted above.  However, $G/Z(G)$ cyclic implies that $G$ is abelian, so
$Z(G)$ must be trivial.

We claim that $G$ has no normal subgroups other than $G$ and $\{e\}$.  Let $H$ be a normal subgroup of $G$ with $H \neq G$, and let $d$ denote the order of $H$.
Since $H$ is proper in $G$, $H$ must be cyclic. Let $c: G\rightarrow Aut(H)$ be the homomorphism induced by the 
conjugation action of $G$ on $H$.  The first isomorphism theorem for groups tells us that $G/ker(c)$ is isomorphic to a subgroup of $Aut(H)$, so
its order divides $\phi(d)$.  On the other hand, certainly $|G/ker(c)|$ divides $n$, so since $(n, \phi(d)) = 1$, we must have 
$|G/ker(c)| = 1$.  Thus the kernel of the homomorphism $c$ is the entire group $G$, so since $c$ is conjugation, we have that
every element of $G$ commutes with every element of $H$.  Therefore, $H \subseteq Z(G)$, but since $Z(G)$ is trivial, $H$ must be trivial as 
well.  Thus, if $H \triangleleft G$ and $H \neq G$, it must be that $H = \{e\}$. 

We also claim that the intersection of any two distinct proper maximal subgroups of $G$ is trivial.  Let $K$ and $L$ denote two different
proper maximal subgroups of $G$, and consider their intersection subgroup $K \cap L$, and its centralizer $C_G(K \cap L)$.
By assumption, $K$ and $L$ are both cyclic, so both abelian, implying that $K$ and $L$ both centralize $K \cap L$.  Therefore, 
$C_G(K \cap L)$ contains both $K$ and $L$, but since $K$ and $L$ are maximal and distinct, we can only have $C_G(K \cap L) = G$. 
Thus, $K \cap L \subset Z(G)$, but since $Z(G)$ is trivial, so is $K \cap L$.  One immediate corollary of this
is that the nonidentity elements of $G$ are partitioned by the nonidentity elements of the distinct maximal subgroups of $G$, that is,
every element in $G\setminus\{e\}$ is in exactly one of the $M_i\setminus\{e\}$, where $M_1,M_2,..., M_k$ are the maximal subgroups of $G$.

Now, we consider the action of conjugation by elements of $G$ on the set of its maximal subgroups.  Let $k$ be the number of orbits 
created by the conjugation action, and let $M_i$ denote a representative of each orbit, for $1 \leq i \leq k$.  For notational simplicity, let
$ord(M_i) = m_i$.  For each $M_i$ representing an orbit, we certainly have that the stabilizer of $M_i$ contains $M_i$.  However, the 
stabilizer of $M_i$ cannot be $G$, because we proved above that we cannot have $M_i$ normal in $G$.  Since $M_i$ is maximal, we can only have
that $M_i$ is its own stabilizer, implying that the orbit it represents has order $n/m_i$.

Therefore, we can use the partition established above to write that 
$$n - 1 = \sum\limits_{i = 1}^k (n / m_i)(m_i - 1)$$
because the number of nonidentity elements of $G$ equals the number of nonidentity elements in the distinct maximal subgroups of $G$.
Dividing by $n$, we obtain
$$ 1 - 1/n = \sum\limits_{i = 1}^k (1 - 1/m_i)$$
and expanding the summation,
$$ 1 - 1/n = k - \sum\limits_{i=1}^k 1/m_i$$
If $k = 1$, then $n = m_1$, impossible.  If $k \geq 2$, note that because each $m_i > 1$, we have $1/m_i \leq 1/2$.  Thus, after rearranging,
$k - 1 + 1/n \leq k/2 \implies 2k - 2 + 2/n \leq k \implies k + 2/n \leq 2$, so $k < 2$, a contradiction.  Therefore, our desired counterexample group $G$
cannot exist, and the theorem is proven.

\end{proof}

\section{Abelian Numbers}

Before trying to prove which positive integers $n$ are abelian numbers, we give some preliminary results.

\subsection[Nonabelian Groups of Order the Cube of a Prime]{Existence of Nonabelian Groups of Order the Cube of a Prime}

\begin{lemma}\label{lem:prime-cubed}

For every prime number $p$, there exists a nonabelian group of order $p^3$.

\end{lemma}

\begin{proof}

Consider the set of $3 \times 3$ matrices with entries in $\mathbb{F}_{p}$ of the form 
$
\begin{pmatrix}

1 & a & b \\
0 & 1 & c \\
0 & 0 & 1 

\end{pmatrix}
$
, where $a$, $b$, and $c$ are arbitrary elements of $\mathbb{F}_{p}$.
It is not hard to show that this is a group of order $p^3$ under matrix multiplication.  On the other hand,
we have that 
$
\begin{pmatrix}

1 & 1 & 1 \\
0 & 1 & 1 \\
0 & 0 & 1

\end{pmatrix}
\begin{pmatrix}

1 & 1 & 0 \\
0 & 1 & 0 \\
0 & 0 & 1

\end{pmatrix}
=
\begin{pmatrix}

1 & 2 & 1 \\
0 & 1 & 1 \\
0 & 0 & 1

\end{pmatrix}
\neq
\begin{pmatrix}

1 & 2 & 2 \\
0 & 1 & 1 \\
0 & 0 & 1

\end{pmatrix}
=
\begin{pmatrix}

1 & 1 & 0 \\
0 & 1 & 0 \\
0 & 0 & 1

\end{pmatrix}
\begin{pmatrix}

1 & 1 & 1 \\
0 & 1 & 1 \\
0 & 0 & 1

\end{pmatrix}
$
, so this group is not abelian.

\end{proof}

\subsection{Burnside's Normal Complement Theorem} 

This section is dedicated to proving the following famous theorem \cite[p.203]{hall}:

\begin{theorem}[Burnside's Normal Complement Theorem] \label{thm:burn}

Let $G$ be a finite group, and let $P$ be a Sylow subgroup of $G$ such that 

$$ P \subseteq Z(N_G(P))$$.

Then there exists $H \triangleleft G$ such that $|G : H| = |P|$.

\end{theorem}

To prove Burnside's Normal Complement Theorem, we first establish the concept of the \emph{group transfer}
\cite[p.200]{hall}.
Recall that $H'$ denotes the commutator subgroup of $H$.

\begin{definition}

Let $G$ be a group, and let $H$ be a (not necessarily normal) subgroup of $G$ with finite index $[G:H] = n$.  
Let $X = \{x_1,x_2,\dots,x_n\}$ be a specific choice of representatives of the right cosets $Hx$ of $H \backslash G$, and define a 
mapping $\phi: G \rightarrow X$ as $\phi(g) = x_i$ if $g \in Hx_i$.  Then the \emph{transfer of 
$G$ into $H$} is the mapping $V_{G \rightarrow H}: G \rightarrow H/H'$ defined by
$$V_{G \rightarrow H}(g) = \prod_{i=1}^n x_ig\phi(x_ig)^{-1} \text{ mod } H'.$$    

\end{definition}

Note that here ``$\alpha$ mod $H'$" means a coset of $H/H'$ containing $\alpha$, and
$\alpha \cong \beta$ mod $H'$ means that $\alpha\beta^{-1} \in H'$.

The motivation for the group transfer comes from the theory of monomial permutations
and monomial representations of groups, as the elements $x_ig\phi(x_ig)^{-1}$ in the product
mimic elements of $H$ in a transitive monomial representation of $G$ with multipliers in $H$.
More on monomial representations can be found in \cite[p.200]{hall}.

We prove essential properties of the group transfer:

\begin{theorem} \label{thm:trans-prop}

\begin{enumerate}
\item $V_{G \rightarrow H}$ is a homomorphism of $G$ into $H/H'$.
\item The transfer $V_{G \rightarrow H}$ is independent of the choice 
$\{x_1,x_2,\dots,x_n\}$ of coset representatives, that is, the value of
$V_{G \rightarrow H}(g)$ is always the same regardless of the choice of $x_i$s.

\end{enumerate}

\end{theorem}       

\begin{proof}

Throughout this proof, we will use the fact that elements of $H/H'$ commute with each other,
since $H/H'$ is an abelian group known as the abelianization of $H$.

(1) Let $g_1$ and $g_2$ be elements of the group $G$, and fix a set of representatives 
$\{x_1,x_2,\dots,x_n\}$ of the right cosets in $H \backslash G$.  For each $x_i$ with $1 \leq i \leq n$, we find the corresponding
$x_j$ such that $x_ig_1 \in Hx_j$, and pick $h_{ij} \in H$ such that $x_ig_1 = h_{ij}x_j$.  Note that
every $Hx_j$ for $1 \leq j \leq n$ is attained by $x_ig_1$ for some $i$, since otherwise there would exist
$x_{k_1} \neq x_{k_2}$ such that $x_{k_1}g_1 \cong x_{k_2}g_1 \text{ mod } H$, or $x_{k_1} \cong x_{k_2} 
\text{ mod } H$, clearly impossible.
 
Now, using the above set of $x_j$s with $1 \leq j \leq n$, we can pick $h_{js} \in H$ such that 
$x_jg_2 = h_{js}x_s$ for some representative $x_s$.  Now,  
$$V_{G \rightarrow H}(g_1g_2) = \prod_{i=1}^n x_ig_1g_2\phi(x_ig_1g_2)^{-1}  \text { mod } H' \cong $$
$$\prod_{i=1}^n h_{ij}h_{js}x_s\phi(h_{ij}h_{js}x_s)^{-1} \text{ mod } H' \cong 
\prod_{i=1}^n  h_{ij}h_{js} \text{ mod } H',$$
while 
$$V_{G \rightarrow H}(g_1)V_{G \rightarrow H}(g_2) = 
\prod_{i=1}^n x_ig_1\phi(x_ig_1)^{-1} \prod_{j=1}^n x_jg_2\phi(x_jg_2)^{-1} \text{ mod } H' \cong $$
$$\prod_{i=1}^n h_{ij}x_j\phi(h_{ij}x_j)^{-1} \prod_{j=1}^n h_{js}x_s\phi(h_{js}x_s)^{-1} \text{ mod } H' \cong $$
$$\prod_{i=1}^n h_{ij} \prod_{j=1}^n h_{js} \text{ mod } H' \cong 
\prod_{i=1}^n h_{ij}h_{js} \text { mod } H',$$ where the last equality
follows from $H/H'$ being abelian.  Note that the product of the $h_{ij}h_{js}$ in terms of only
$i$ is well-defined, since each $h_{js}$ depends only on the 
fixed $g_2$ and on the index $j$, which in turn depends
only on $i$ and the fixed $g_1$.

(2)  Let us fix $g \in G$, and take two different sets of representatives for the right 
cosets of $H$ as $\{x_1,x_2,\dots,x_n\}$ and $\{y_1,y_2,\dots,y_n\}$.  For each 
$1 \leq i \leq n$, let $a_i \in H$ be the element such that $y_i = a_ix_i$.  Now, for each 
$x_i$, suppose that $x_ig = h_{ij}x_j$ as above.  Then $y_ig = a_ix_ig = 
a_ih_{ij}x_j = a_ih_{ij}a_j^{-1}y_j$.

We evaluate the transfer using $\{x_1,x_2,\dots,x_n\}$ as representatives:
$$V_{G \rightarrow H}(g) = \prod_{i=1}^n x_ig\phi(x_ig)^{-1} \text { mod } H' \cong $$
$$ \prod_{i=1}^n x_igx_j^{-1} \text{ mod } H' \cong \prod_{i=1}^n h_{ij} \text{ mod } H'.$$
Then we evaluate the transfer using $\{y_1,y_2,\dots,y_n\}$ as representatives:
$$V_{G \rightarrow H}(g) = \prod_{i=1}^n y_ig\phi(y_ig)^{-1} \text{ mod } H' \cong $$
$$\prod_{i=1}^n y_igy_j^{-1} \text{ mod } H' \cong \prod_{i=1}^n a_ih_{ij}a_j^{-1} \text{ mod } H' \cong $$
$$\prod_{i=1}^n a_i \prod_{j=1}^n a_j^{-1} \prod_{i=1}^n h_{ij} \text{ mod } H' \cong
\prod_{i=1}^n h_{ij} \text{ mod } H',$$
since the set of $a_j$s is just a reordering of the set of $a_i$s.

\end{proof}

We will need one further definition and lemma:

\begin{definition}
A \emph{complex} is an arbitrary set of elements of a group $G$ that need not form a subgroup.  However, as a 
set it can still be normal in $G$ or have conjugates in $G$.
\end{definition}
 
\begin{lemma}

Let $G$ be a group with a Sylow subgroup $P$, and let $K_1$ and $K_2$ be complexes of $P$.  Suppose
that $K_1$ and $K_2$ are both normal in $P$, and are conjugate to each other in $G$.  Then they are
also conjugate to each other in $N_G(P)$.

\end{lemma}

\begin{proof}

By conjugation, we can find $x \in G$ such that $xK_1x^{-1} = K_2$.  Since $K_1 \triangleleft P$
we must have $xK_1x^{-1} \triangleleft xPx^{-1}$, so $K_2 \triangleleft Q$ for the 
Sylow subgroup $Q = xPx^{-1}$ of $G$.  Now $P$ and $Q$ are Sylow subgroups contained in
$N_G(K_2)$, so are conjugate in $N_G(K_2)$.  Therefore, there exists some $y$ with 
$yK_2y^{-1} = K_2$ that also satisfies $yQy^{-1} = P$.  Now, the element $yx \in N_G(P)$ since
$yxP(yx)^{-1} = yxPx^{-1}y^{-1} = yQy^{-1} = P$.  Furthermore, 
$yxK_1(yx)^{-1} = yxK_1x^{-1}y^{-1} = yK_2y^{-1} = K_2$, so $K_1$ is conjugate to $K_2$
by an element of $N_G(P)$ as desired.

\end{proof}

Now the proof of Burnside's Complement Theorem: 

\begin{proof}

To show the existence of a normal complement group $H$, it suffices to find an 
onto homomorphism $\phi: G \rightarrow P$, for then we can take $ker \phi = H$.  
If $P \subseteq Z(N_G(P))$, then $P \subseteq Z(P)$ and $P$ is abelian, so $P' = \{e\}$.  Thus, the
the transfer $V_{G \rightarrow P}: G \rightarrow P/P'$ may be treated as a homomorphism
$V_{G \rightarrow P}: G \rightarrow P$.  We demonstrate that the transfer is onto by showing that it 
maps $P$ onto itself isomorphically, i.e. if we restrict the transfer to a homomorphism 
$P \rightarrow P$, it is in fact an isomorphism.  Thus, we will only evaluate the transfer on elements of
$P$.    
 
Let $|G : P| = n$, and let $X = \{x_1,x_2,\dots,x_n\} $ be some choice of representatives of the right cosets
of $P \backslash G$.  Since the transfer is independent of the choice of representatives, 
for any particular $u \in P$, we will construct a particular set $S_u$ of representatives as follows: 
Pick some $x_i \in X$, and construct the elements $x_i,x_iu, \dots, x_iu^{r_i-1}$ 
until reaching the least $r_i$ such that $x_iu^{r_i} \in Px_i$.  Add the elements $x_i, x_iu, \dots, x_iu^{r_i-1}$ to $S_u$.  
If this does not exhaust all right cosets of $P \backslash G$, pick a representative $x_j \in X$ whose coset is not already represented in 
$S_u$, and repeat.  The final set $S_u$ of representatives for $u$ contains elements 
forming disjoint cycles across the right cosets $Px$.  Note that the set $S_u$ is closed under right multiplication by $u$, so the homomorphism $\phi$ will be easy to evaluate.  

Now, using the set $S_u = \{s_1,s_2,\dots,s_n\}$ as representatives, $V_{G \rightarrow P}(u) = \prod_{i=1}^n s_iu\phi(s_iu)^{-1}$ (we exclude the mod 
$P'$ since $P' = \{e\}$).  Let us look at a subset of $S_u$ containing one of the disjoint cycles 
$x_i,x_iu, \dots, x_iu^{r_i-1}$, where $x_iu^{r_i} \in Px_i$, and $x_iu^j \not\in Px_i$ for $0 \leq j < r_i$.  
Then for any $0 \leq k < r_i-1$, the corresponding factor in the transfer is 
$(x_iu^k)u\phi(x_iu^{k+1})^{-1} = x_iu^{k+1}(x_iu^{k+1})^{-1} = e$, while for $k = r_i-1$, the factor is
$(x_iu^{r_i-1})u\phi(x_iu^{r_i})^{-1} = x_iu^{r_i}x_i^{-1}$, so each disjoint cycle of length $r_i$ in our set $S_u$ of representatives 
contributes only a factor of $x_iu^{r_i}x_i^{-1}$ and the rest are identity elements.  Now, the transfer can be simplified 
to $V_{G \rightarrow P}(u) = \prod_i x_iu^{r_i}x_i^{-1}$, where the product ranges across the disjoint cycles of $S_u$.

Now we will apply the above lemma.  We consider the one-element complex $K_1 = \{x_iu^{r_i}x_i^{-1}\}$.  By construction 
$x_iu^{r_i} \in Px_i$, so $x_iu^{r_i}x_i^{-1} \in P$.  This complex is conjugate in $G$ to the complex $K_2 = \{u^r\}$, which
also clearly lies in $P$.  Furthermore, since $P$ is abelian, both complexes are normal in $P$.  By the above lemma, there 
exists $y \in N_G(P)$ such that $x_iu^rx_i^{-1} = yu^ry^{-1} = u^r$ since $P \subseteq Z(N_G(P))$.  

Now the transfer at $u$ is
$V_{G \rightarrow P}(u) = \prod_i u^{r_i} = u^n$, since the sum of the lengths of the disjoint cycles of $S_u$ is just the 
size of $S_u$, which is also $|G : P|$.  This applies to all $u \in P$, so the restriction of $V_{G \rightarrow P}$ to 
$P$ is equivalent to the homomorphism $\phi: P \rightarrow P$ given by $\phi(u) = u^n$.  

Now, since $P$ is a Sylow $p$-subgroup of $G$ for some prime $p$, the index $n$ is relatively prime to $p$.  Thus
$u^n \neq e$ unless $u = e$, so $\phi$ is injective.  Since $\phi$ is an injective homomorphism between sets of 
equal finite cardinality, it must also be an isomorphism.  Now the transfer $V_{G \rightarrow P}$ is necessarily 
a surjective homomorphism of $G$ onto $P$, so its kernel is a normal subgroup of $G$ of order $|G : P|$.
      
\end{proof}

\subsection{Determining the Abelian Numbers}

\begin{theorem}

A positive integer $n$ is an abelian number $\iff$ 
$$ n = p_1^{a_1} \dots p_r^{a_r}, \text{ where each }
a_i \leq 2, \text{ and  } (p_i,p_j^{a_j} - 1) = 1 \text{ when } i \neq j.$$

\end{theorem}

The following proof is adapted from \cite{haines2}.

\begin{proof}

We prove that the given condition is necessary.  Note that it is necessary for $a_i \leq 2$ for each exponent, because
we established in Lemma~\ref{lem:prime-cubed} that for every prime $p$ there are always nonabelian groups of order $p^3$.  We also have that 
$(p_i, p_j - 1) = 1$ for all $i \neq j$ since we established in the proof of cyclic numbers that if
$p_i \mid p_j - 1$, there is a nontrivial semidirect product of order $p_ip_j$, and nontrivial semidirect
products are never abelian.  

Let $p_k$ be a prime factor of $n$ with $a_k = 2$.  For any other prime $p_i$ that is a factor of $n$, we cannot have that 
$p_i \mid p_k + 1$.  If that were the case, then $p_i$ would divide the order of 
$Aut(C_{p_k} \times C_{p_k}) = p_k(p_k - 1)^2(p_k + 1)$, so there would again be a nonabelian group of order $p_ip_k^2$,
and so a nonabelian group of order $n$.  Since $(p_i, p_k - 1) = 1$ and $(p_i, p_k + 1) = 1$, we must have that
$(p_i, p_k^{a_k} - 1) = 1$ when $a_k = 2$.  Thus, any abelian number $n$ must satisfy the conditions of the theorem.

We now show that the condition is sufficient.  Let $n$ be the minimal positive integer
satisfying our conditions that is not an abelian number, and let $G$ be a nonabelian group of order $n$.  Any divisor
$d$ of $n$ must also satisfy our conditions, so all proper subgroups and quotient groups of $G$ must be abelian.

We prove that $G$ is solvable.  For any prime $p$ dividing $n$, take $P$ to be a Sylow $p$-subgroup of $G$.  
If $P \triangleleft G$ then both $P$ and $G/P$ are abelian, thus solvable, so $G$ is solvable.  If $P$ is not
normal in $G$, then its normalizer $N_G(P) \neq G$, so $N_G(P)$ must be abelian (as a group of order less than $n$), 
and clearly $P \subseteq Z(N_G(P))$.  By Theorem \ref{thm:burn}, $P$ must have a nontrivial 
normal complement $K$ in $G$.  But then $K$ and $G/K$ are both solvable, and so $G$ is solvable as desired.    

Consider the commutator subgroup $G'$ of $G$.  If $G' = G$, then $G$ would not be solvable, so $G' \neq G$.
Also, $G' \neq \{e\}$, for otherwise we would have $G$ abelian.  Thus, $G'$ is a proper, nontrivial subgroup of $G$.  
Take $Q$ to be a nontrivial Sylow $q$-subgroup of $G'$ for some prime $q$ dividing $n$.  Since $G'$ is abelian,
we must have $Q \text{ char } G'$ and of course $G' \triangleleft G$, so $Q \triangleleft G$.  Now, since $Q$ is taken to
be nontrivial, $G/Q$ is abelian, and so $Q$ must contain $G'$.  Since also $G'$ contains $Q$ by construction,
we must have that $Q = G'$. 

Now, let $Q_0$ be a Sylow $q$-subgroup of the entire group $G$ that contains $Q = G'$.  We may assume that
$Q_0 \neq G$; otherwise, $G$ is a group of order $q$ or $q^2$, so is certainly abelian.  Theorem \ref{thm:burn}
tells us that there exists some subgroup $K$ of $G$ satisfying $ord(K) = ord(G)/ord(Q_0)$.  We consider the 
conjugation action of $G$ on $Q$, restricted to the subgroup $K$.  Since $Q \triangleleft G$, this conjugation is a 
homomorphism $c: K \rightarrow Aut(Q)$.  However, because $n$ contains only prime factors with multiplicity of at most $2$,
$Q$ can only be isomorphic to $C_q$,  $C_{q^2}$, or $C_q \times C_q$ for some prime $q$, so $ord(Aut(Q))$ can only be 
$q - 1$, $q(q-1)$, or $q(q-1)^2(q+1)$ respectively.  However, if $q^2 \mid n$, it is clear that $ord(K)$ consists by definition only of 
primes that do not divide any of $q$, $q-1$, or $q+1$.  Otherwise, $Q$ is isomorphic to $C_q$, and still any prime in 
$ord(K)$ fails to divide $q-1$.  It follows that the conjugation action $c$ is trivial, and therefore the elements of $K$
commute with the elements of $Q$.

We see that $Q \subseteq Z(Q_0)$, since $Q_0$ is abelian and $Q \subseteq Q_0$.
Also, $K \cap Q_0 = \{e\}$, and $ord(K)ord(Q_0) = n$, so we must have that $KQ_0 = G$.  Therefore, $Q \subseteq Z(G)$, so
$Z(G)$ is nontrivial.  Also, $Z(G) \neq G$ since $G$ is not abelian.  However, these combined mean that $G/Z(G)$ is abelian
by assumption, so $G$ is nilpotent.  This means that $G$ is a direct product of its Sylow subgroups (see Section~\ref{sec:nilpotent}).  However,
every Sylow subgroup of $G$ has order that is either prime or the square of a prime, so they are all cyclic groups
or a direct product of cyclic groups.  Thus, if $G$ is a direct product of its Sylow subgroups, it is a direct product
of cyclic groups, and so is abelian, a contradiction. 

\end{proof}

One nice corollary of knowing whether a given $n$ is an abelian number is that we immediately know how many groups
there are of order $n$ based only on its prime factorization.  If $n = p_1^{a_1} \dots p_r^{a_r}$ is abelian, then the 
number of groups of order $n$ is 
$$ \prod_{i=1}^r 2^{a_i - 1} $$
This is because all finite abelian groups are a direct product of cyclic groups, and we therefore get that any group of order $n$
would be a direct product across all primes $p_i$ of exactly $C_{p_i}$ if $a_i = 1$, or of a choice of
$C_{p_i} \times C_{p_i}$ and $C_{p_i^2}$ if $a_i = 2$.   

\section{Nilpotent Numbers} \label{sec:nilpotent}

In this section, we will establish some properties of finite nilpotent groups, use some of these properties to
determine the set of nilpotent numbers, and then show how the sets of abelian and cyclic numbers can be determined
as a corollary of the theorem on nilpotent numbers, providing an alternate proof to those found
in previous sections.

\subsection{Properties of Nilpotent Groups}

\begin{definition}

For any group $G$ (not necessarily finite in this case), let the \emph{upper central series} $Z_0(G) \subseteq Z_1(G) 
\subseteq Z_2(G) \subseteq \dots $ of $G$ be given recursively by
$$ Z_0(G) = \{e\}, Z_1(G) = Z(G), \text{and } Z_{i+1}(G)/Z_i(G) = Z(G/Z_i(G)) $$

The group $G$ is called \emph{nilpotent} if $Z_c(G) = G$ for some term in the upper central series, and its 
\emph{nilpotency class} is the index $c$ first giving equality.

\end{definition}

Note that in particular, a group with nilpotency class 1 is abelian.

Since we are dealing with only finite groups, we establish some equivalent conditions to nilpotency in these groups.

\begin{theorem} \label{thm:nilpotent-prop}

Let $G$ be a finite group, whose order is divided by distinct primes $p_1, p_2, \dots p_r$, and containing
Sylow subgroups $P_1, P_2, \dots P_r$ corresponding to those primes.  Then the following conditions are equivalent:

\begin{enumerate}
\item $G$ is nilpotent
\item Every proper subgroup of $G$ is a proper subgroup of its normalizer, i.e. if $H \subsetneq G$, then 
$H \subsetneq N_G(H)$
\item Every Sylow subgroup is normal in $G$, i.e. $P_i \triangleleft G$ for each $1 \leq i \leq r$.
\item $G$ is the direct product of its Sylow subgroups, i.e. $G = P_1 \times P_2 \times \dots \times P_r$
\end{enumerate} 

\end{theorem}

This proof is adapted from \cite[p. 191]{dummit}.

\begin{proof}

We prove that $1 \implies 2$ by induction on the nilpotency class of $G$.  If $G$ has nilpotency class 1, then 
it is abelian, and then every subgroup of $G$ is normal, and the result holds.  Now, suppose that the claim holds for 
all groups of nilpotency class $c$, and consider a group $G$ of nilpotency class $c+1$.  Let $H$ be any subgroup
of $G$.  We certainly must have that $Z(G) \subseteq N_G(H)$.  Therefore, if $Z(G)$ is not contained in $H$, 
then $H \subsetneq N_G(H)$.  Thus, we may suppose that $Z(G) \subseteq H$.  Now, we pass to the quotient 
$G/Z(G)$, which is nilpotent with class $c$.  We have that $H/Z(G)$ is a proper subgroup of $G/Z(G)$, so by the 
inductive hypothesis, $H/Z(G) \neq N_{G/Z(G)}(H/Z(G)) = N_G(H)/Z(G)$, so indeed $H \neq N_G(H)$, and the
induction is complete.
	
To prove that $2 \implies 3$, let $P$ be any Sylow subgroup of $G$, and let $N = N_G(P)$ be its normalizer in $G$.
By definition $P \triangleleft N$, and since $P$ is a Sylow subgroup, this says that $P$ is characteristic in $N$, so since
$N \triangleleft N_G(N)$, we have $P \triangleleft N_G(N)$.  Thus, 
$N_G(N)$ is contained in the normalizer of $P$, so $N_G(N) \subseteq N$, so obviously $N_G(N) = N$.  However, 
from $(2)$ we know that no proper subgroup of $G$ is its own normalizer, so we can only have $N = G$, which 
implies that $P \triangleleft G$ as desired.

To prove that $3 \implies 4$, we prove inductively that 
$P_1P_2 \dots P_t \cong P_1 \times P_2 \times \dots \times P_t$ for any $1 \leq t \leq r$.  The base case is
trivial.  Suppose that $P_1P_2 \dots P_k \cong P_1 \times P_2 \times \dots \times P_k$
for some $k < r$.  Certainly $P_1P_2 \dots P_{k+1}$ is a subgroup, as each $P_i$ is normal in $G$. Let 
$H = P_1P_2 \dots P_k$ and let $K$ = $P_{k+1}$.  By induction, $H \cong P_1 \times P_2 \times \dots \times P_k$.
Clearly the orders of $H$ and $K$ are relatively prime, and Lagrange's Theorem implies that $H \cap K = \{e\}$,
so the result now follows.

To prove that $4 \implies 1$, proceed by induction on the order of $G$.  The base case is trivial, so let $G$ 
be a minimal counterexample satisfying $G \cong P_1 \times P_2 \times \dots \times P_r$ that is not nilpotent.
We have that 
$$Z(G) = Z(P_1 \times P_2 \times \dots \times P_r) \cong Z(P_1) \times Z(P_2) \times \dots \times Z(P_r)$$
and so 
$$ G/Z(G) = P_1/Z(P_1) \times P_2/Z(P_2) \times \dots \times P_r/Z(P_r).$$
Since no $p$-group has trivial center, $Z(G)$ is nontrivial, and we may apply the inductive hypothesis to 
conclude that $G/Z(G)$ is nilpotent, from which it follows that $G$ is nilpotent.

\end{proof}

\subsection{Nilpotent Numbers}

Before determining the nilpotent numbers, we will need a theorem of O.J. Schmidt's regarding a certain type 
of finite group \cite[p. 258]{robinson}:

\begin{theorem}

Let $G$ be a finite group that is not nilpotent, but suppose that every proper maximal subgroup of $G$ is nilpotent.  Then the
order of $G$ is divisible by exactly two distinct primes, i.e. $ord(G) = p^mq^n$ with $m,n > 0$
for distinct primes $p$ and $q$.

\end{theorem}

\begin{proof}

We first prove that if a group is not nilpotent but has every proper maximal subgroup nilpotent, then it is solvable.  Suppose not, and take $G$ to be a minimal counterexample.
If $G$ has any proper nontrivial normal subgroup $N$, then $N$ is nilpotent, hence solvable.  Furthermore, $G/N$ is either nilpotent and thus solvable, or
still satisfies that all of its maximal subgroups are nilpotent, in which case it is solvable by the minimality of $G$ as a counterexample.
However, if both $N$ and $G/N$ are solvable, then so is $G$.  Therefore, we may assume that $G$ is simple.

Now we suppose that every pair of distinct maximal subgroups of $G$ intersect in only the identity.  Let $M$ 
be a maximal subgroup of $G$.  Because $G$ is simple, $M = N_G(M)$.  Let $n = ord(G)$, and let $m = ord(M)$.
Then $M$ has $n/m$ conjugates in $G$, all of which intersect pairwise in only the identity.  Therefore, the 
nonidentity elements of the conjugates of $M$ account for a total of $(m-1)n/m = n - n/m$ 
nonidentity elements of $G$.  Since $m \geq 2$, we must have $n - n/m \geq n/2 > (n-1)/2$.  Also, we certainly have
$n - n/m \leq n - 2 < n - 1$.  

By assumption, each nonidentity element of $G$ is contained in exactly one 
maximal subgroup of $G$.  We may thus count the nonidentity elements of $G$
by taking representatives $M_1, M_2, \dots, M_k$ of the $k$ distinct sets of conjugates of the maximal subgroups of $G$.
Let $ord(M_i) = m_i$.  Then
$$ n - 1 =  \sum\limits_{i=1}^k n - n/m_i$$
However, if $k = 1$, we cannot have all $n-1$ nonidentity elements accounted for, while if $k > 1$, the right side
of the equation is already larger than $n-1$, a contradiction.  

Thus, there is some pair of distinct maximal subgroups of $G$ that intersect nontrivially.  Choose maximal subgroups
$M_1$ and $M_2$ of $G$ so that their intersection $I$ is as large as possible.  Let $N$ denote $N_G(I)$.  $M_1$ is
nilpotent by assumption, so $I \neq N_{M_1}(I)$ (Theorem ~\ref{thm:nilpotent-prop}), meaning that $I \subsetneq N \cap M_1$.  We also already established
that $I$ cannot be normal in $G$, so its normalizer $N$ must be contained in some maximal subgroup $M_0$ of $G$.
We now have that $I \subsetneq N \cap M_1 \subseteq M_0 \cap M_1$, but this contradicts the maximality of $I$.
This last contradiction implies that $G$ is solvable.

Now, given that $G$ is solvable, we can prove that the order of $G$ contains exactly two distinct prime factors.
Let the order of $G$ have standard prime factorization $p_1^{a_1}p_2^{a_2} \dots p_r^{a_r}$, and proceed by contradiction
by assuming that $r \geq 3$.  Take $M$ to be a maximal normal subgroup of $G$.  It is well-known that a maximal
normal subgroup of a solvable group has prime index, so w.l.o.g. let $|G : M| = p_1$.  Let $P_i$ be a 
Sylow $p_i$-subgroup of $G$.  If $i > 1$, then $P_i \subseteq M$.  Since $M$ is nilpotent, we have
$P_i \text{ char } M$, and by assumption $M \triangleleft G$, so $P_i \triangleleft G$.  Now, for each $i > 1$, 
$P_1P_i$ is a subgroup of $G$, but it is not all of $G$ if $r \geq 3$.  Therefore, the group $P_1P_i$ is nilpotent, and the elements of $P_1$
and $P_i$ commute.  Since this holds for every $P_i$ with $i > 1$, we have that $N_G(P_1) = G$, so
$P_1 \triangleleft G$.  But now every Sylow subgroup of $G$ is normal in $G$, so $G$ is nilpotent, a contradiction.

\end{proof}

We also give the following characterization of nilpotent numbers its own name, so that it may
be referenced easily in the following subsection.

\begin{definition}

Let $n$ be a positive integer with standard prime factorization $p_1^{a_1}p_2^{a_2} \dots p_r^{a_r}$.  The integer $n$ is 
said to have \emph{nilpotent factorization} provided that $p_i \nmid p_j^k - 1$, for any $i \neq j$ and
any $1 \leq k \leq a_j$. 

\end{definition}

Now we are ready to prove the proper characterization of all nilpotent numbers.

\begin{theorem}

A positive integer $n$ is a nilpotent number $\iff$ $n$ has nilpotent factorization.

\end{theorem}

The following proof is adapted from \cite{pakianathan}.

\begin{proof}

We first prove necessity.  Let $n$ be a positive integer with standard prime factorization 
$p_1^{a_1}p_2^{a_2} \dots p_r^{a_r}$, and suppose $n$ does not have nilpotent factorization.  Then there exist
primes $p_i$ and $p_j$ and an exponent $k$ between $1$ and $a_j$ such that $p_i \mid p_j^k - 1$.  For simplicity,
relabel so that we have $p_2 \mid p_1^k - 1$.  Let $E$ be an elementary abelian group consisting of the 
direct product of $k$ copies of $C_{p_1}$, which we denote by $C_{p_1}^k$.
We can determine the order of $Aut(E)$ by identifying it with 
$GL_k(\mathbb{F}_{p_1})$.  This is the set of $k \times k$ matrices with entries in $\mathbb{F}_{p_1}$ and nonzero determinant.  By
constructing such a matrix row-by-row, we can see that the order of $GL_k(\mathbb{F}_{p_1})$ and thus of $Aut(E)$ is 
$(p_1^k - 1)(p_1^k - p_1) \dots (p_1^k - p_1^{k-1})$.  By assumption, $p_2 \mid ord(Aut(E))$, so there exists a 
nontrivial semidirect product $E \rtimes C_{p_2}$.  We construct the group of order $n$
$$ (E \rtimes C_{p_2}) \times C_{p_1}^{a_1 - k} \times C_{p_2}^{a_2 - 1} \times C_{p_3}^{a_3} \times \dots \times C_{p_r}^{a_r}$$
In a nilpotent group, elements from different Sylow subgroups commute, but in our group, all of the elements of 
$E$ have order $p_1$, and they do not commute with the elements of order $p_2$ in the semidirect product.  Therefore,
we have produced an example of a group of order $n$ that is not nilpotent.

Now we prove sufficiency.  Let $n$ be a minimal counterexample, so $n$ has 
nilpotent factorization but is not a nilpotent number.  Take a group $G$ of order $n$ that is not nilpotent.
Any divisor of $n$ also has nilpotent factorization, so any proper subgroup of $G$ is nilpotent.  Since $G$
is a finite group that is not nilpotent but has all subgroups nilpotent, we may apply Theorem 5.3
to conclude that $n = p^aq^b$ for distinct primes $p$ and $q$.  

Let $n_p$ and $n_q$ denote the number of Sylow $p$-subgroups and Sylow $q$-subgroups of $G$ respectively.  
By Sylow's theorem, $n_p \equiv 1 \text{ (mod }p)$, and $n_p$ also equals the index in $G$ of the normalizer of
some Sylow $p$-subgroup, $N_G(S_p)$.  Since $S_p \subseteq N_G(S_p) \subseteq G$, we know that the order of 
$N_G(S_p)$ is $p^aq^k$ for some integer $k$, and has index in $G$ equal to $q^{b-k}$.  Thus, 
$q^{b-k} \equiv 1 \text{ (mod }p)$, but since $n$ has nilpotent factorization, this is only possible if $b - k = 0$, so that
$N_G(S_p) = G$.  Thus, $S_p \triangleleft G$, and a symmetric argument on $q$ shows that $S_q \triangleleft G$.  
But this means that all Sylow subgroups are normal in $G$, and so $G$ is nilpotent, a contradiction. 

\end{proof}

\subsection[Applying Nilpotent Factorization]{Applying Nilpotent Factorization to Abelian and Cyclic Numbers}

The nilpotent factorization that characterized nilpotent numbers easily extends to provide an alternate
proof of the characterization of abelian and cyclic numbers.

\begin{theorem}

A positive integer $n$ is an abelian number $\iff$ $n$ is cube-free and has nilpotent factorization.

\end{theorem}  

\begin{proof}

Let $n$ be a cube-free nilpotent number, and let $G$ be any group of order $n$.  Then $G$ is nilpotent, and so
is the direct product of its Sylow subgroups.  However, if $n$ is cube-free, each of the Sylow subgroups has order that is a
prime or the square of a prime, so each Sylow is the direct product of one or two cyclic groups, and $G$ is also a direct 
product of cyclic groups, hence abelian.

Now, assume that $n$ is an abelian number.  This implies that it is also a nilpotent number, since all abelian groups
are nilpotent.  We also know that $n$ is cube-free, since there exist nonabelian groups of order $p^3$ for any prime $p$,
so $n$ is a cube-free nilpotent number.   

\end{proof}

An analogous proof shows that the cyclic numbers can be characterized as the square-free nilpotent numbers.

\section{Ordered Sylow numbers}

\subsection{Ordered Sylow Towers}

\begin{definition} \label{def:ordsyl}

Let $n$ be a positive integer with standard prime factorization $p_1^{a_1}p_2^{a_2} \dots p_r^{a_r}$, where
$p_1 < p_2 < \dots < p_r$. Let $G$ be a group of
order $n$, and let $P_i$ denote a Sylow $p_i$-subgroup of $G$.
The group $G$ is said to have an \emph{ordered Sylow tower} if any of the following equivalent conditions are satisfied:
\begin{enumerate}
\item $G$ has a normal series $G = R_0 \supset R_1 \supset \dots \supset R_r = \{e\}$ such that 
$|R_{i-1} : R_i| = p_i^{a_i}$ for $i = 1, 2, \dots, r$.
\item For every $i = 1, 2, \dots, r$, the product $P_iP_{i+1} \dots P_r$ is normal in $G$, regardless of the choice of 
Sylow subgroups.
\item For every $i = 1, 2, \dots, r$, there exists some Sylow $p_i$-subgroup such that in the set
$\{P_1, P_2, \dots, P_r\}$, every $P_i$ is normalized by its predecessors, i.\ e.\ $P_i \triangleleft P_1 \dots P_i$ for all
$1 \leq i \leq r$.
\end{enumerate} 
\end{definition}

It is not inherently obvious that these definitions are in fact identical, so we prove their equivalence.
We will need the following lemma:

\begin{lemma} \label{lem:index}

If a group $G$ has subgroups $H$ and $K$ of relatively prime index, 
then $|G : H \cap K| = |G : H| \times |G : K|$.

\end{lemma}

\begin{proof}
Note that for any $H$ and $K$, $|G : H \cap K| = |G : H| \times |H : H \cap K|$ and certainly
$|G : H \cap K| \leq |G : H| \times |G : K|$.  However, if $|G : H|$ and $|G : K|$ are relatively prime, then because 
$|G : K|$ must divide $|G : H \cap K|$, it must divide $|H : H \cap K|$, implying that $|G : H \cap K| \geq |G : H| \times |G : K|$,
forcing $|G : H \cap K| = |G : H| \times |G : K|$.
\end{proof}

We will also need another theorem of Philip Hall:

\begin{theorem} \label{thm:hall-set}

If $G$ is a finite solvable group of order $p_1^{a_1} \dots p_r^{a_r}$, then there exists
a particular set of Sylow subgroups
$P_1, P_2, \dots, P_r$ such that $P_iP_j = P_jP_i$ is a group for all $i \neq j$, and in fact
$P_{i_1}P_{i_2} \dots P_{i_k}$ is a group for any subset $\{i_1, i_2, \dots, i_r\}$ of the indices 
from $1$ to $r$.

\end{theorem}

This proof is taken from \cite[p. 232]{gorenstein}. 

\begin{proof}

Let $p_i'$ denote the set of all distinct primes dividing the order of $G$ except $p_i$.  
Then Hall's Theorem (Theorem ~\ref{thm:hall}) tells us that there exist subgroups $H_1, H_2, \dots, H_r$ of $G$
satisfying $|G : H_i| = p_i^{a_i}$.  For any $i \neq j$, $|G : H_i|$ is relatively prime to 
$|G : H_j|$, so $|G : H_i \cap H_j| = |G : H_i| \times |G : H_j|$ by Lemma ~\ref{lem:index}.  Therefore, 
$H_i \cap H_j$ is a Hall $\pi$-subgroup of $G$ for $\pi$ the set of all primes dividing the order of
$G$ except $p_i$ and $p_j$.  This can continue to show that $H_i \cap H_j \cap H_k$ is a 
Hall subgroup of $G$ for all distinct $i$,$j$, and $k$, and in general we can produce all possible Hall
subgroups as intersections of the $H_i$s.

Eventually we create Sylow subgroups $P_1, \dots, P_r$ from this process.  
For this set of Sylow subgroups, $P_i$ is the intersection of all the $H$s except for $H_i$. 
For $i \neq j$, let $Q_{ij}$ be the Hall $\{p_ip_j\}$-subgroup created by intersecting all of the $H$s except for $H_i$ and $H_j$. 
Then $P_iP_j \subseteq Q_{ij}$ and $ord(P_iP_j) = ord(Q_{ij})$, so $P_iP_j = P_jP_i = Q_{ij}$ is a subgroup of 
$G$, and similarly so are arbitrary products of members of this set of Sylow subgroups. 

\end{proof}

\begin{lemma}

The three conditions in Definition \ref{def:ordsyl} are all equivalent.  

\end{lemma}

\begin{proof}

To show $2 \implies 1$, proceed by reverse induction from $r$ to $1$.  Definition $2$ says that
$P_r \triangleleft G$, so we may take $R_{r-1} = P_r$ for the base case.  Now, if for some $k > 1$ we have
$P_k\dots P_r = R_{k-1}$, then since $P_{k-1} \dots P_r \triangleleft G$, the third isomorphism theorem
tells us that $R_{k-1}$ is normal in $P_{k-1} \dots P_r$ and has index $p_{k-1}^{a_{k-1}}$, so we may take
$R_{k-2} = P_{k-1} \dots P_r$, and thus construct a normal series of $G$ satisfying definition $1$.

To show $1 \implies 2$, we again proceed by reverse induction from $r$ to $1$.  In the normal series
of definition $1$, $R_{r-1}$ must be a Sylow $p_r$-subgroup of $G$, but then certainly this Sylow
subgroup $P_r$ is normal (and unique) in $G$, so the base case is established.  Now let $k > 1$, 
and suppose that for any choice of the Sylow subgroups, 
$P_k \dots P_r = R_{k-1}$ for every normal series of $G$ satisfying definition $1$, which also
implies that $P_k \dots P_r \triangleleft G$.  Then there is some $R_{k-2}$ such that
$R_{k-1} \triangleleft R_{k-2}$, and $|R_{k-2} : R_{k-1}| =   p_{k-1}^{a_{k-1}}$.  Certainly $R_{k-2}$
contains some Sylow $p_{k-1}$-subgroup $P_{k-1}^*$ of $G$, so we must have that 
$R_{k-2} = P_{k-1}^*P_k \dots P_r$, and therefore $P_{k-1}^*P_k \dots P_r \triangleleft G$ for this choice
$P_{k-1}^*$.  However, any general choice  $P_{k-1}$ of the Sylow $p_{k-1}$-subgroup in $G$ produces the same $R_{k-2}$,
because any two Sylow $p_{k-1}$-subgroups
are conjugate, and $(gP_{k-1}^*g^{-1})P_k \dots P_r = gP_{k-1}^*(g^{-1}P_k \dots P_r) = 
gP_{k-1}^*(P_k \dots P_rg^{-1}) = g(P_{k-1}^*P_k \dots P_r)g^{-1} = P_{k-1}^*P_k \dots P_r$.
 
We now show that $1$ and $2 \implies 3$.  First we show that Definition 1 implies that $G$ is solvable.  Indeed, $G = R_0$ has a 
normal subgroup $R_1$ such that $R_0 / R_1$ is a $p$-group, so is solvable.  Then $G$ is solvable if $R_1$ is solvable.  
But $R_1$ has a normal subgroup $R_2$ such that $R_1 / R_2$ is solvable, so $G$ is solvable if $R_2$ is solvable, and continuing this 
process, we determine that $G$ is solvable if the trivial group is solvable, so $G$ is solvable. We may thus 
apply Theorem \ref{thm:hall-set} to find a set
of Sylow subgroups $S_1, \dots, S_r$ such that $S_1 \dots S_i$ is a subgroup of $G$ for all $1 \leq i \leq r$.  In 
this set of Sylow subgroups, since $S_1 \dots S_i$ is a subgroup for any $i$ and $S_i \dots S_r \triangleleft G$ by
Definition 2, the second isomorphism theorem tells us that $S_1 \dots S_i \cap S_i \dots S_r = S_i$ is a normal subgroup
of $S_1 \dots S_i$, which is equivalent to definition 3.

Finally, to show that $3 \implies 2$, we again proceed by reverse induction from $r$ to $1$.  Let the specific
Sylow subgroups in definition 3 be given by $S_1,S_2, \dots, S_r$.  The base case is clear,
since the specific $S_r$ in definition 3 is normal in $G$, and unique.
Now, for $k > 1$, we suppose that for any choice of the Sylow subgroups $P_k, P_{k+1}, \dots, P_r$, the product $P_k \dots P_r$ is identical to $S_k \dots S_r$, 
and that $S_k \dots S_r$ is normal in $G$.  Then certainly $G / S_k \dots S_r \cong S_1 \dots S_{k-1}$, so 
definition 3 implies that $\overline{S}_{k-1} \triangleleft G / S_k \dots S_r$ for $\overline{S}_{k-1}$ 
isomorphic to $S_{k-1}$.  This implies that
the group $S_{k-1}S_k \dots S_r$ is normal in $G$, but since all Sylow
$p_{k-1}$-subgroups are conjugate, we may again conclude that $S_{k-1}S_k \dots S_r$ is identical to 
$P_{k-1}S_k \dots S_r$ for any choice of $P_{k-1}$, which is identical to any possible $P_{k-1} \dots P_r$ and is normal in $G$.

\end{proof}

If the above definition is satisfied for some permutation of the primes $p_i$ that is not in increasing order,
then $G$ is said to simply have a \emph{Sylow tower}.

To establish results in this section, we will make use of the multiplicative $\psi$ function, defined on prime powers as
$\psi(p^k) = (p^k - 1)(p^{k-1} - 1) \dots (p - 1)$.

\subsection{Preliminary Results}

\begin{definition}

Let $G$ be a finite group, and let $B$ be the intersection of every normal subgroup $M$ of $G$ such that
$G/M$ is nilpotent.  Then the group $B$ is normal in $G$.  Furthermore, the standard homomorphism
from $G/B$ into the direct product of all nilpotent factor groups $G/M_i$ of $G$
is injective by definition, so $G/B$ is isomorphic to a subgroup of this direct product.  Since this direct product is nilpotent,
$G/B$ is nilpotent, and is called the \emph{maximal nilpotent factor group} of $G$.  

Similarly, let $p_i$ be a prime dividing the order of $G$, and let $A_i$ be the intersection of every normal subgroup
of $G$ that has $p_i$-power index in $G$.  Then $A_i$ is normal and has $p_i$-power index in $G$, and $G/A_i$ is called the
\emph{maximal $p_i$-factor group} of $G$.

\end{definition}

We will need the following lemma from \cite[p. 129]{zassenhaus} relating these kinds of subgroups:

\begin{lemma} \label{lem:max}

The maximal nilpotent factor group of a finite group $G$ is isomorphic to the direct product of 
the maximal $p_i$-factor groups of $G$ for each prime $p_i$ dividing its order.

\end{lemma}

\begin{proof}

Let $B$ be the intersection of all normal subgroups $M$ of $G$ where $G/M$ is nilpotent, and let $A$ be the intersection of 
all of the $A_i$ in $G$ such that $G/A_i$ is a maximal $p_i$-factor group of $G$.  We first prove that $A = B$.

Certainly every normal subgroup $N$ of prime power index in $G$ satisfies that $G/N$ is nilpotent, so
$B \subseteq A$ because $B$ is an intersection of at least all the subgroups intersected in $A$.
To prove that $A \subseteq B$, we prove that if a group element $c \notin B$, then 
$c \notin A$.  If $c \notin B$, then there is some normal subgroup $K$ of $G$ such that $G/K$ is nilpotent and 
$c \notin K$. Let $p_0$ be a prime number dividing the order of $\bar{c}$, the image of $c$ under
the standard homomorphism $h_1 : G \rightarrow G/K$.  
Since $G/K$ is the direct product of its Sylow subgroups, we may compose standard homomorphisms 
that quotient out each Sylow subgroup of $G/K$ except for $P_0$, the Sylow $p_0$-subgroup.  This creates
a homomorphism $h_2$ from $G/K$ to a $p_0$-group.
Because the order of $\bar{c}$ is divisible by $p_0$, it does not map to the identity in the 
image of $h_2$, so $\bar{c}$  is not in the kernel of $h_2$.  Now, this means that $c$ is not in the kernel of
the homomorphism $h_2 \circ h_1$, but this kernel is a normal subgroup of $G$ of $p_0$-power index,
so $c \notin A$, and $A = B$.

To prove that $G/A \cong G/A_1 \times G/A_2 \times \dots \times G/A_r$, we first note that the standard homomorphism 
$h : G/A \rightarrow G/A_1 \times G/A_2 \times \dots \times G/A_r$ is injective by definition.  It will therefore suffice to prove
that $G/A$ and $G/A_1 \times G/A_2 \times \dots \times G/A_r$ have the same order, so we must prove that
the index $|G : A|$ is equal to the product of the indices $|G : A_1| \times |G : A_2| \times \dots \times |G : A_r|$.  

Now let $A_{m_*}$ represent $A_1 \cap A_2 \cap \dots \cap A_m$.  We prove by induction that 
$|G : A_{m_*}| =  |G : A_1| \times |G : A_2| \times \dots \times |G : A_m|$ for all $1 \leq m \leq r$, from which it would follow that 
indeed $|G : A| = |G : A_1| \times |G : A_2| \times \dots \times |G : A_r|$.  The base case is trivial, so assume that for
some $k < r$ we have $|G : A_{k_*}| = |G : A_1| \times |G : A_2| \times \dots \times |G : A_k|$.  Because $|G : A_{k_*}|$
and $|G : A_{k+1}|$ must be relatively prime, we have by Lemma \ref{lem:index} that 
$|G : A_{k_*} \cap A_{k+1}| = (|G : A_1| \times |G : A_2| \times \dots \times |G : A_k|) \times |G : A_{k+1}|$,
and the induction holds. 

\end{proof}

We wish to prove a further results on maximal $p$-factor groups, but
we first need to derive the two \Grun theorems.  To derive these, we will need to 
prove the following theorem:

\begin{theorem} \label{thm:trans-P}

Suppose that $P$ is a Sylow $p$-subgroup of $G$, and $G'$ the commutator subgroup of $G$.  Then
$V_{G \rightarrow P}(G) \cong P/(P \cap G')$.

\end{theorem}

The proofs for this theorem and the two \Grun theorems are adapted from \cite[p.213]{hall}.

\begin{proof}

First, note that $V_{G \rightarrow P}$ is a homomorphism of $G$ into $P/P'$, which is a $p$-group, so
any element of $G$ of order relatively prime to $p$ must be mapped to the identity in $P/P'$.  Since $G$ is generated by $P$
along with other Sylow subgroups containing elements of order relatively prime to $p$, it suffices to consider the transfer $V$ restricted to $P$.  Take any set of representatives $\{x_1,\dots,x_n\}$ of the conjugacy classes of $P \backslash G$.  To evaluate on a particular $u \in P$, we consider the same construction $S_u$ of disjoint cycles of elements of the form 
$\{x_i, x_iu,\dots,x_iu^{r-1}\}$ used in the proof of Burnside's Complement Theorem.  Following the same argument as in that proof, we may simplify the transfer to $V(u) = \prod_i x_iu^{r_i}x_i^{-1}$ mod $P'$, 
where the $i$ range across the disjoint cycles.  Continuing, we find that
$\prod_i x_iu^{r_i}x_i^{-1} \text{ mod } P' = \prod_i u^{r_i}(u^{-r_i}x_iu^{r_i}x_i^{-1}) $ mod $P'$.  This does not immediately simplify mod $P'$ as $x_i \not\in P$, but we may calculate this expression further mod $G'$ and find that $V(u) = \prod_i u^{r_i} \text{ mod } G' = u^n \text{ mod } G'$. 

By definition $n$ and $p$ are relatively prime, so for any $u \not\in G'$, we must have $V(u) \neq 1$ mod $G'$.
However, $V(G)$ is always abelian, so $V(G') \cong 1 \text{ mod } P'$ necessarily.  It follows that the kernel of our restricted transfer on $P$ is $P \cap G'$, so $V_{G \rightarrow P}(G) \cong P / (P \cap G')$.   
\end{proof}

Now we will proceed to the two \Grun theorems.

\begin{theorem}[First Theorem of Gr\"un] \label{thm:grun1}

Let $P$ be a Sylow $p$-subgroup of a group $G$.  Then $V_{G \rightarrow P}(G) \cong P/P^*$, where
$$P^* = [P \cap (N_G(P))']\cup_{z \in G} (P \cap z^{-1}P'z).$$

\end{theorem}

Note: The notation $G \cup H$ of groups $G$ and $H$ represents the group generated 
by the elements of $G$ and $H$.  In fact, the basic set-theoretic union of distinct $G$ and $H$ is never a group, since elements of the form $gh$ with $g \in G, h \in H$ and 
$g \not\in H, h \not\in G$ can never lie in either $G$ or $H$. 

\begin{proof}

By \ref{thm:trans-P}, we have that $V_{G \rightarrow P}(G) \cong P/(P \cap G')$.  It is clear that 
$[P \cap (N_G(P))'] \subseteq (P \cap G')$ and $(P \cap z^{-1}P'z) \subseteq (P \cap G')$, so that
$P^* \subseteq P \cap G'$.  It suffices to show that $P \cap G' \subseteq P^*$.  We prove that every $u \in P \cap G'$ also lies 
in $P^*$ by induction on the order of $u$ 
(note that this order is $p^r$ for some $r \geq 0$ since we are only considering $u \in P$).   
For order 1, it is clear that the identity element lies in both subgroups.  

We will now use a double coset construction $P \backslash G/P$.  In such a construction, $a \sim b$ 
if and only if there exist $p_1,p_2 \in P$ such that $a = p_1bp_2$.  Let $Y = \{y_1,y_2,\dots,y_s\}$ be a set of representatives of $P \backslash G/P$.  For a given $y_i$, let us further write the coset $Py_iP$ as $Py_iv_1 \cup Py_iv_2 \cup \dots \cup Py_iv_k$ for $v_1 = e$ and some choices of $v_k \in P$. 
Then we have a complete set of left cosets of $G$ of the form $Py_iv_k$.  Let us fix $u \in P$ for evaluation. 
We will once again use a set of representatives $S_u$ as was initially given in the Burnside Complement Theorem, so that in this case cycles of $S_u$
are of the form $y_iv_k, y_iv_ku, \dots, y_iv_ku^{r_{ik}-1}$, where 
$y_iv_ku^{r_{ik}} \in Py_iv_k$.  

We may evaluate $V(u)$ on each double coset separately, that is, we will fix $y_i$ 
and evaluate on all cosets of the form $Py_iv_k$.  The transfer evaluated on the double coset
$Py_iP$ be given by $w_i = \prod_k y_iv_ku^{r_{ik}}v_k^{-1}y_i^{-1}$.  Note that $\sum_k r_{ik} = p^t$ for some $t \geq 0$, since this is the number of left cosets of the form $Py_i$ in the double coset $Py_iP$.  We evaluate $w_i$ in two cases:

Case 1: $t \geq 1$.  Then $w_i = y_iu^{p^t}y_i^{-1}$ when restricted mod $y_iP'y_i^{-1}$.  In the particular case of $v_1 = e$,
we have the factor $y_iu^{p^b}y_i^{-1} \in P$, $b \leq t$.  We also must have that $w_i \in P$ by the construction of the integers $r_{ik}$, so 
in fact $w_i \cong y_iu^{p^t}y_i^{-1} \text{ mod } P \cap y_iP'y_i^{-1}$.  But this group $P \cap y_iP'y_i^{-1}$ lies in $P^*$ by assumption, so $w_i \cong y_iu^{p^t}y_i^{-1} \text{ mod } P^*$.

Now, by assumption $u \in P \cap G'$, so $V(u) \cong 1 \text{ mod } P'$, 
so clearly $V(y_iu^{p^t}y_i^{-1}) \cong 1 \text{ mod } P'$.  
However, now $y_iu^{p^t}y_i^{-1}$ is in the kernel of our transfer, $P \cap G'$, and since $t > 1$, it has order smaller than that of $u$.  By the inductive hypothesis, $y_iu^{p^t}y_i^{-1} \in P^*$.  Also, $u^{p^t} \in P^*$ by the inductive
hypothesis, so $w_i = y_iu^{p^t}y_i^{-1} \cong 1 \cong u^{p^t} \text{ mod } P^*$.

Case 2: $t = 0$.  Then $Py_iP = Py_i$ as cosets, so $Py_i \subseteq N_G(P)$.  Furthermore, 
$w_i = y_iuy_i^{-1} = u(u^{-1}y_iuy_i^{-1}) \cong u $ mod $(N_G(P))'$, and so
$w_i \cong u$ mod $P \cap (N_G(P))'$.   Since $P \cap (N_G(P))' \subseteq P^*$, we have that
$w_i \cong u = u^{p^t} \text{ mod } P^*$.

In any case, $w_i \cong u^{p^{t_i}} \text{ mod } P^*$ for all $i$.  Then the total evaluation of $V(u)$ is
$\prod_i w_i = \prod_i u^{p^{t_i}}$, and this is just $u^{\sum_i p^{t_i}}$.  But recall that 
$p^{t_i}$ is the number of left cosets of $Py_iP$, so the sum of these is the total number $n$ of left cosets of $P$ in $G$.  Thus, $V(u) \cong u^n \text{ mod } P^*$.  But recall that $V(u) \cong 1 \text{ mod } P'$ when $u \in P \cap G'$, so $V(u) \in P' \subseteq P^*$.  But then
$u^n \cong 1$ mod $P^*$, and $n$ is relatively prime to $p$, so $u \in P^*$ as desired.  
 
\end{proof}

Before moving on, we introduce a definition:

\begin{definition}

A finite group $G$ is \emph{$p$-normal} for a prime $p$ if whenever there exists a Sylow $p$-subgroup $P$ of $G$
such that another Sylow $p$-subgroup $Q$ of $G$ contains $Z(P)$, $Z(P) = Z(Q)$.  

\end{definition}

\begin{theorem}[Second Theorem of Gr\"un] \label{thm:grun2}

Suppose that $G$ is a $p$-normal group for some prime $p$.  Then the maximal abelian $p$-factor group of $G$ is isomorphic
to the maximal abelian $p$-factor group of $N_G(Z(P))$ for a Sylow $p$-subgroup $P$.    

\end{theorem}

Note again that the union of a set of subgroups is defined to be the subgroup
generated by them.

\begin{proof}

Let $G'$ be the commutator subgroup of $G$, and let $G'(p) \supseteq G'$ be the smallest normal subgroup of $G$ such that
$G/G'(p)$ is an abelian $p$-group.  The order of $G'(p)$ must contain every factor of the order of $G$ except for
$p$-powers, so we may conclude that $G = G'(p) \cup P$.  Let $G^* = G' \cup P$.  Clearly then $G^* \cup G'(p) = G$.  
Now, $G^* / G' \subseteq P$, so contains only elements of $p$-power order, while $G'(p)/G'$ by definition contains only elements of order relatively prime to $p$, so the intersection $G^* \cap G'(p)$ must be precisely $G'$.  Now, it is clear that
$(G'(p) \cup G^*) / G'(p)) \cong (G^* / (G^* \cap G'(p))$, which means that
$G/G'(p) \cong G^*/G' \cong P / (P \cap G')$. 

Now let $Z = Z(P)$, $N = N_G(P)$ and $H = N_G(Z)$. 
Since the center $Z$ is characteristic in $P$, it must be that $N \subseteq H$.
Let $H'(p)$ denote the smallest normal subgroup of $H$ such that $H/H'(p)$ is an abelian $p$-group. 
Then using the same argument as above, we may conclude that $H/H'(p) = P / (P \cap H')$.  Since $H$ is the normalizer of the
center of a Sylow $p$-subgroup, we would like to show that $G/G'(p) = H/H'(p)$, or equivalently 
$P \cap G' = P \cap H'$.  Since obviously $P \cap H' \subseteq P \cap G'$, we must prove that
$P \cap G' \subseteq P \cap H'$.

Using the First Theorem of Gr\"un (\ref{thm:grun1}), we know that 
$$P \cap G' = (P \cap N') \cup_{g \in G} (P \cap g^{-1}P'g).$$
We already know that $N \subseteq H$, so $P \cap N' \subseteq P \cap H'$, and it remains to show that
$P \cap g^{-1}P'g \subseteq P \cap H'$ for all $g \in G$.

For a given $g \in G$, let $M = P \cap g^{-1}P'g$.  Then certainly $Z \subseteq N_G(M)$, and since $g^{-1}Zg$ is the center of
$g^{-1}Pg$, we also have $g^{-1}Zg \subseteq N_G(M)$.  Let $R$ be a Sylow $p$-subgroup of $N_G(M)$ containing $Z$, and let
$S$ be a Sylow $p$-subgroup of $N_G(M)$ containing $g^{-1}Zg$.  There is some $y \in N_G(M)$ such that $Z$ and 
$y^{-1}x^{-1}Zxy$ are both in $R$.  Let $Q$ be a Sylow $p$-subgroup of $G$ containing $R$.  Since $G$ is $p$-normal and $Q$ contains $Z(P)$ for a Sylow $p$-subgroup $P$ of $G$, it must be that $Z(P) = Z(Q)$.  
Thus, $Z$ is the center of $Q$ and also $Z = y^{-1}x^{-1}Zxy$, so $xy \in N_G(Z) = H$.  Let $h = xy$.  Then since
$y \in N_G(M)$, 
$$M = y^{-1}My = y^{-1}Py \cap y^{-1}x^{-1}P'xy = y^{-1}Py \cap h^{-1}P'h \subseteq H'.$$
Then $M \subseteq P \cap H'$, and we are done.

\end{proof}

We will also need results of Burnside and Hall.  First, a lemma of Burnside:

\begin{lemma} \label{lem:burn}

Let $G$ be a finite group, and let $h$ be a $p$-subgroup of $G$ for some prime $p$.  Suppose that there exist two
different Sylow $p$-subgroups containing $h$ such that $h$ is normal in one, but not in the other.  
Then there exists a number $r > 1$ relatively prime to $p$ and subgroups $h_i$, $1 \leq i \leq r$ all conjugate
to $h$ such that:

\begin{itemize}

\item All the $h_i$ are normal in the subgroup $H = h_1h_2 \cdots h_r$.
\item There is no Sylow $p$-subgroup in which all of the $h_i$ are normal.
\item The $h_i$ are a complete set of conjugates of each other in $N_G(H)$.

\end{itemize}

\end{lemma}

The proof is adapted from \cite[p.46]{hall}

\begin{proof}

Since we will use many normalizers in this proof, we will temporarily use the notation $N_h$ to refer to
$N_G(h)$ for concision.  Also, lower case letters (except p) in this proof 
will refer to groups unless otherwise specified for
clarity and ease of notation.    

By assumption at least one Sylow $p$-subgroup of $G$ contains $h$ non-normally.  Among all Sylow $p$-subgroups 
containing $h$ non-normally, choose $Q$ to be one such that $D = N_h \cap Q$ is as large as possible.  Let
$N_Q(D)$ denote the normalizer of $D$ in $Q$, and note that $N_D$ denotes the normalizer of $D$ in $G$. 
Now there is some subgroup of $q \subseteq Q$ containing $h$ such that $|q : h| = p$, which would imply
$h \triangleleft q$, so that $h$ is properly contained $D$. 
Since also $h$ is not normal in $Q$ by assumption, $D$ is properly contained in $Q$.  Now $D$ is a proper subgroup
of the $p$-group $Q$, so it is properly contained in its normalizer in $Q$.  Altogether, we have the relations
$h \subsetneq D \subsetneq N_Q(D) \subseteq Q$.  

Now, $D$ clearly contains $N_Q(h)$, and $D \subsetneq N_Q(D)$, so
$h$ cannot be normal in $N_Q(D)$, and so certainly not in $N_D$.  Thus, there are $s > 1$ conjugates of $h$ in
$N_D$, say $h = h_1, h_2, \dots, h_s$.  Since $h \triangleleft D$, each $h_i \triangleleft D$, and
$H = h_1 \cup \dots \cup h_s \subseteq D$.  Also, elements of $G$ normalizing $D$ certainly map $H$ back to itself, so
$N_D \subseteq N_H$.

Now let $p_1$ denote a Sylow $p$-subgroup of $N_h \cap N_D$, and let $P_1$ be a Sylow $p$-subgroup of $N_h$ that 
contains $p_1$.  It must be that $P_1$ is also a Sylow $p$-subgroup of $G$, 
since there is at least one Sylow $p$-subgroup of $G$ that contains $h$ as a normal subgroup.  
Now, certainly $D$ is properly contained in $P_1$, so that $D$ is not its own normalizer in $P_1$.  We wish to show also
that $D$ is properly contained in $p_1$.  Suppose to the contrary that $D = p_1$.  Let $q_1$ be a Sylow $p$-subgroup of $P$ that properly contains and normalizes $D = p_1$.  
Then $q_1$ is a $p$-group in $N_h \cap N_D$ properly containing $p_1$, contradicting the definition of $p_1$ as a Sylow $p$-subgroup.  It follows that in fact $D$ must be properly contained in $p_1$.  
Furthermore, $N_h \cap N_D \subseteq N_D \subseteq N_H$, so let $p_2$ be a Sylow $p$-subgroup
of $N_H$ containing $p_1$, and let $P$ be a Sylow $p$-subgroup of $G$ containing $p_2$.  Suppose $P$ does not normalize
$h$.  Then $D \subsetneq p_1 \subseteq P \cap N_h$, but this would contradict our choice of $D$ as the maximal intersection
of $N_h$ with a Sylow $p$-subgroup not normalizing $h$.

It follows that $P \subseteq N_h$, and $P \cap N_H \subseteq N_h \cap N_H$, but as $p_2 \subseteq P$ and
$p_2$ is a Sylow $p$-subgroup of $N_H$, $p_2 = P \cap N_H$.  Now, let $h = h_1,h_2, \dots, h_s, \dots, h_r$ be a 
complete set of conjugates of $h$ in $N_H$.  Note that as $h$ is normal in $H$, 
each $h_i$ is contained in and normal in $H$.  Since the 
normalizer of $h$ in $N_H$ is equivalent to $N_H \cap N_h$, the number of conjugates is
$|N_H: N_H \cap N_h|$.  But since $N_H \cap N_h$ contains $p_2$ as a Sylow $p$-subgroup, it must be that
$p$ does not divide $|N_H : N_H \cap N_h|$, so $r$ is relatively prime to $p$.  Also $r \geq s > 1$ as was
proven before.

It remains to show that there is no Sylow $p$-subgroup of $G$ in which all of the $h_i$, $1 \leq i \leq r$ are
normal.  Suppose there were such a Sylow $p$-subgroup $S_p$.  Then by definition 
$S_p \subseteq N_H$, so every Sylow $p$-subgroup of $N_H$ would contain and normalize the $h_i$.  But 
$N_Q(D) \subseteq N_D \subseteq N_H$ is a $p$-group of $N_H$ which fails to normalize $h = h_1$ as previously
noted.                  

\end{proof}

Next, a pair of theorems based on the work of Burnside and Philip Hall, as shown in
\cite[p.176]{hall}:

\begin{theorem}[Burnside Basis Theorem] \label{thm:basis}

Let $P$ be a $p$-group of order $p^n$, and let $D$ be the intersection of the maximal subgroups of $P$ (also
known as the Frattini subgroup of $P$).  Then the quotient group $P/D = A$ is an elementary
abelian group, say of order $p^r$, so is isomorphic to $\mathbb{F}_p^r$.  Furthermore, every 
set of elements $\{z_1,z_2,\dots,z_s\}$ that generates $P$ contains a subset of $r$ elements
$\{x_1,x_2,\dots,x_r\}$ which generates $P$.  In the natural projection homomorphism 
$P \rightarrow P/D = A$, the set $\{x_1,\dots,x_r\}$ is mapped to a basis 
$\{a_1,\dots,a_r\}$ of $A$, and conversely, any set of $r$ elements of $P$ that are projected
onto a set of generators of $A$ must be a generating set for $P$.    

\end{theorem}

\begin{proof}

First, we wish to show that every maximal subgroup of a $p$-group $P$ has index
$p$ and is normal in $P$.  We proceed by induction on the size of the
$p$-group.  In the base case, every group of order $p$ has the trivial group as the only maximal subgroup, and it is certainly normal and of index $p$. 
Now, let $M$ be some maximal subgroup of $P$, and by the inductive hypothesis, suppose that every proper subgroup and quotient group of $P$ satisfies this property.  Let $r$
be an element of order $p$ in $Z(G)$ (such an element must exist since centers of $p$-groups
are nontrivial), and let $R$ be the subgroup generated by $r$.  If $r \in M$, then
$M/R$ is a maximal subgroup of $P/R$, so the inductive hypothesis on $P/R$ tells us that
$M/R$ is normal and of index $p$ in $P/R$, and so $M$ is normal and of index $p$ in $P$.  
If $r \not\in M$, then $MR$ is a larger subgroup than $M$, so by maximality
$MR = P$, but then it is immediate that $M \triangleleft P$ and $M$ is of index $p$.          

Thus, for every maximal subgroup $M$ of $P$, $P/M$ is cyclic of order $p$, so the $p^{th}$ power of
every element of $P$ is in $M$, as is every commutator.  It follows that $D$, the intersection
of all maximal subgroups of $P$, contains every $p^{th}$ power and every commutator in $P$.  
It follows that $P/D$ is abelian, and every element of $P/D$ has order $p$, so $P/D$ is an
elementary abelian group $A$.  

Let $|A| = p^r, r \leq n$.  Then viewing $A$ as isomorphic to $\mathbb{F}_p^r$, 
any basis of $A$ consists of $r$ elements, say $\{a_1,\dots,a_r\}$, and clearly any set of more
than $r$ elements generating $A$ contains a subset of $r$ of them which also generate $A$.  

Let $\{z_1,\dots,z_s\}$ be some set of elements generating $P$.  In the projection
$P \rightarrow P/D = A$, let $z_i$ be mapped to $b_i$ for $1 \leq i \leq s$.  
The set of $b_i$ must generate $A$, so $s \geq r$, 
and some subset $\{a_1,\dots,a_r\}$ also generates $A$.
Let $\{x_1,\dots,x_r\}$ be the corresponding preimages 
from among the $z_i$.  

We wish to show that the elements $x_i$ generate $P$.  
Let $X$ be the subgroup of $P$ generated by the $x_i$. Suppose that $X \neq P$.  Then $X$ is
contained in a maximal subgroup $M$ of $P$.  Then in the projection $P \rightarrow P/D$,
we have that $X \rightarrow X/D \subseteq M/D = B$, where $B$ is some subgroup of $A$ of order
$p^{r-1}$, but this contradicts that the images $a_i$ generate all of $A$.  Thus, $X = P$, and the rest of the theorem immediately follows.   
  
\end{proof}

We use this to prove Philip Hall's theorem on automorphisms of $p$-groups:

\begin{theorem} \label{thm:auto}

Let $P$ be a group of order $p^n$, $D$ the intersection of the maximal subgroups of $P$,
and $|P : D| = p^r$.  Let $A(P)$ be the group of automorphisms of $P$, where the order of an automorphism $\alpha$ is the least $k$ such that $\alpha^k$ is the identity.  Then
$|A(P)|$ divides $p^{r(n-r)}\theta(p^r)$, where
$$\theta(p^r) = (p^r - 1)(p^r - p) \cdots (p^r-p^{r-1}).$$

\end{theorem} 

\begin{proof}

From Theorem \ref{thm:basis}, we know that $P/D = A$ is an elementary abelian $p$-group,
of order $p^r$, so the number of ways to choose a basis of $A$ is
$(p^r-1)(p^{r-1}-1)\dots(p-1) = \theta(p^r)$.  Let $\{a_1,\dots,a_r\}$ be a fixed basis of $A$ . 
Then any mapping of this basis to another basis of $A$ yields an automorphism of $A$, but
since every automorphism of $A$ must map the set $\{a_1,\dots,a_r\}$ to some basis, there must
be exactly $\theta(p^r)$ automorphisms of $A$.

We wish to choose $X = \{x_1,\dots,x_r\}$ to be
an ordered set of $r$ elements that generate $P$.  In the
projection $P \rightarrow A$, we may choose a basis $\{a_1,\dots,a_r\}$ for $A$ to map $X$
onto in $\theta(p^r)$ ways, and for each $a_i$, there are $p^{n-r}$ 
choices of $x_i$ as an element in the
coset of $D$ corresponding to $a_i$.  Thus, for each $x_i$, there are $p^{n-r}$ choices,
and we make such a choice $r$ times, so there are $(p^{n-r})^r$ total choices of the $x_i$ after
picking a basis, so there are $p^{r(n-r)}\theta(p^r)$ total choices of the set $\{x_1,\dots,x_r\}$.
But the group $A(P)$ of automorphisms is a permutation group on the set of $X$s, so its order must
divide $p^{r(n-r)}\theta(p^r)$ as desired.  

\end{proof}

Finally, we will use the Second \Grun theorem and Philip Hall's automorphism theorem to 
prove the following theorem of Frobenius \cite[p. 143]{zassenhaus}:

\begin{theorem} \label{thm:frob}

Let $G$ be a group of order $n$, and let $p^a$ be the largest power of a prime number $p$ that divides $n$.
Suppose that $(n, \psi(p^a)) = 1$.  Then the maximal $p$-factor group of $G$ is isomorphic to every 
Sylow $p$-subgroup of $G$.

\end{theorem}

\begin{proof}

We proceed by induction on the power $a$, so the inductive hypothesis is that for a fixed $a$ 
we assume that every group whose Sylow $p$-subgroup is of order less than $p^a$ satisfies the theorem.  In
the base case where $a = 0$, both are the trivial group $\{e\}$ since no proper subgroup of $G$ 
can have $p$-power index.

Now assume the inductive hypothesis for all $k < a$, and consider a group $G$ with Sylow $p$-subgroup of order 
$p^a$.  We first prove that $G$ is $p$-normal.  If not, there exists a Sylow $p$-subgroup $P$ 
with center $Z(P) \triangleleft P$ and another Sylow $p$-subgroup $Q$ containing $Z(P)$ such that
$Z(P)$ is not normal in $Q$.  Then by Lemma \ref{lem:burn}, there exists a complete set 
of $q > 1$ ($q$ relatively prime to $p$) conjugates 
$Z_1 = Z(P), Z_2, \dots, Z_q$ of $Z(P)$ all normal in $Z_1\cdots Z_q$, 
such that elements in the normalizer of $Z_1\cdots Z_q$ acts transitively on the $Z_i$.  Thus, 
$q$ is the order of an orbit under an action of the normalizer of $Z_1 \cdots Z_q$, so it is also
a divisor of the order of that normalizer, and thus also of $n$.  Furthermore, $q$ is a divisor of
the order of the group of automorphisms of $Z_1 \cdots Z_q$, 
but this is a $p$-group, say of order $p^c$.  
By Hall's Theorem on automorphisms (\ref{thm:auto}), since $q$ is relatively prime to $p$,
it must divide $p^c-1$, and then also $\psi(p^c)$, 
and further $\psi(p^{a})$.  Since $q$ must divide $n$ and $\psi(p^a)$, and
$(n,\psi(p^a)) = 1$ by assumption, we must have $q = 1$, a contradiction.  Thus, $G$ is $p$-normal.

Now, let $P$ be a Sylow $p$-subgroup of $G$.  Suppose that $P$ is abelian.  Then transforming $P$ by
elements of its normalizer induces an automorphism whose order divides $\psi(p^a)$
(by Theorem \ref{thm:auto}) and $n$, so this order
must be 1.  Then $P \subseteq Z(N_G(P))$, and we may apply Burnside's Theorem to complete the proof.

From now on, we may assume $P$ is not abelian.  
Let $Z = Z(P)$.  Then $Z \neq P$, and since $P$ is a $p$-group
$Z \neq \{e\}$, so we may apply our inductive hypothesis
to the group $N_G(Z)/Z$, because this is a group whose Sylow
$p$-subgroup is certainly of lower order than that of $G$.  
We may conclude then that the maximal $p$-factor group of
$N_G(Z)/Z$ is isomorphic to one of its Sylow $p$-subgroups, and in particular is nontrivial.  
By lifting this from $N_G(Z)/Z$ to $N_G(Z)$, we may see that the maximal $p$-factor group of $N_G(Z)$ is nontrivial.  It follows that $N_G(Z)$ has a nontrivial abelian $p$-factor group.
By the Second Theorem of \Grun (\ref{thm:grun2}), the maximal abelian $p$-factor group of 
$G$ is isomorphic to that of $N_G(Z)$, so in particular $G$ has a nontrivial $p$-factor group.  Let $D(G)$ be the intersection of all normal subgroups of $G$ with $p$-power index, 
so that the maximal $p$-factor group of $G$
is $G/D(G)$.  We have shown that $G/D(G)$ is nontrivial.  
If $G/D(G)$ is not isomorphic to some Sylow $p$-subgroup of $G$, then
it must be that $p$ divides $|D(G)|$.  
By the inductive hypothesis, it would follow that $D(G)$ has a maximal $p$-factor group isomorphic to one of its Sylow
$p$-subgroups, so that $D(G)/D(D(G))$ is nontrivial.  But this is absurd, as clearly it must be the case that
$D(G) = D(D(G))$.  Thus, $G/D(G)$ must be isomorphic to a Sylow $p$-subgroup of $G$ as desired.          

\end{proof}

We use this Frobenius theorem \ref{thm:frob} and Lemma \ref{lem:max} to prove the following: 

\begin{lemma} \label{lem:n1}

Let $G$ be a group of order $n$, where $n = n_1n_2$, $(n_1, n_2) = 1$, and $(n, \psi(n_1)) = 1$.  Then the maximal
nilpotent factor group of $G$ has an order divisible by $n_1$.

\end{lemma}

\begin{proof}

Let $F$ denote the maximal nilpotent factor group of $G$, and recall that $F$ is isomorphic to the direct product 
of the maximal $p_i$-factor groups of $G$.  For each prime factor $p_i$ of $n_1$, let $p_i^a$ be the largest power of $p_i$
dividing $n_1$.  Then by assumption $p_i^a$ is the maximal power of $p_i$ dividing $n$, and $(n, \psi(p_i^a)) = 1$, so
the maximal $p_i$-factor group of $G$ has order $p_i^a$.  Therefore, $p_i^a \mid ord(F)$, and applying this to all prime factors
of $n_1$, we have that $n_1 \mid ord(F)$.

\end{proof}

\subsection{Determining the Ordered Sylow numbers}

We are now ready to characterize the ordered Sylow numbers:

\begin{theorem} \label{thm:ordSyl}

Let $n$ be a positive integer with standard prime factorization $p_1^{a_1}p_2^{a_2} \dots p_r^{a_r}$, where 
$p_1 < p_2 < \dots < p_r$.
Then all groups of order $n$ have an ordered Sylow tower $\iff$
$$ (p_i^{a_i} \dots p_r^{a_r}, \psi(p_i^{a_i})) = 1$$
for all $1 \leq i \leq r$.

\end{theorem}

The proof is adapted from \cite[p. 334]{pazderski}.
 
\begin{proof}

We first prove necessity.  Suppose that $n$ is a positive integer that does not satisfy our desired conditions.  
This would mean that $n$ has prime factors $p_j > p_k$ such that $p_j \mid \psi(p_k^{a_k})$.  In particular, this means that
$p_j \mid p_k^m - 1$ for some $1 \leq m \leq a_k$.  We know that this means that we can take the elementary abelian group 
$E = C_{p_k}^m$, and construct a nontrivial semidirect product $G = E \rtimes C_{p_j}$ with order $p_jp_k^m$.  
We established in the characterization of 
nilpotent numbers that this group is not nilpotent.  However, certainly $E$ is a Sylow $p_k$-subgroup of $G$, and $E \triangleleft G$.
We therefore cannot have any Sylow $p_j$-subgroup normal in $G$, else $G$ would be nilpotent.  However, $p_j$ is the largest
prime factor dividing the order of $G$, so if $G$ fails to have a normal Sylow $p_j$-subgroup, it cannot have an ordered 
Sylow tower, and neither could a group of order $n$ containing it.  

Now, suppose that $n = p_1^{a_1}p_2^{a_2} \dots p_r^{a_r}$ satisfies the conditions of the theorem, and let $G$ be a group of order $n$.  Set $G = R_0$.  
Then by Lemma \ref{lem:n1}, there exists a normal subgroup $R_1$ of $R_0$ such that $R_0/R_1$ is nilpotent
and $|R_0 : R_1| = p_1^{a_1}$.  Similarly, there exists $R_2$ a normal subgroup of $R_1$ such that $R_1/R_2$ is nilpotent
and $|R_1 : R_2| = p_2^{a_2}$, and so on.  Furthermore, since all of the indices are relatively prime, all of these 
subgroups are normal in $G$, so we have produced an ordered Sylow tower for $G$.

\end{proof}

\section{Supersolvable Numbers}

In this section, we define and establish basic properties of supersolvable groups,
and then determine a set of criteria on positive integers $n$ such that every group of
order $n$ is supersolvable.  Unlike the criteria established for properties in previous
sections, the set of criteria for $n$ to be a supersolvable number is neither simple to
describe nor particularly intuitive, so our proof of its correctness will be somewhat 
more scattered and require closer attention than the proofs in previous sections.

\subsection{R'edei's First-Order Non-Abelian Groups} \label{sec:redei}

The information in this section is taken from the article 
\emph{Das ``schiefe" Produkt in der Gruppentheorie} (literally \emph{``Skew" Products in Group Theory}) by Laszlo R'edei.

A group is called \emph{first-order non-abelian} if it is not abelian, but every proper subgroup of it is.  R'edei
gave a complete characterization of all finite first-order non-abelian groups in his article in terms of so-called
``skew" products of a group and a finite field.  R'edei gave three examples of these products and demonstrated that
all such groups (except for the quaternion group of order 8) is of one of those three forms.  We will examine his first
skew product, which is the one Pazderski used.  From this point on it will just be referred to as 
\emph{the} skew product for simplicity. 

We define the skew product $RG$ of a group $G$ and a commutative ring $R$ with identity 
to be the set of ordered pairs $(a,\alpha)$, with
$a \in R$ and $\alpha \in G$, where the group operation is defined by 
$(a,\alpha)(b,\beta) = (a+h(\alpha)b, \alpha\beta)$ where $h: G \rightarrow R^{\times}$ is a homomorphism from
$G$ to the multiplicative group of units of $R$.  Thus, $RG$ defined this way consists of the standard group operation
of $G$ in the second coordinate, and the standard addition operation of $R$ in the first coordinate ``skewed" by some
unit corresponding to $\alpha$.  The identity element is $(0,e)$, and the inverse
of $(a,\alpha)$ is $(-h(\alpha)^{-1}a, \alpha^{-1})$.  It is a semidirect product $R^+ \rtimes G$ of the additive group
of $R$ with $G$, and in the case that $h$ is the trivial homomorphism, this skew product is exactly
the direct product $R^+ \times G$.  

In particular, to construct a first-order non-abelian group with this skew product, you should take $G$ to be cyclic of 
order $p^u$ for some choice of prime $p$ and exponent $u$, and you should take $R$ to be the finite field of
$q^v$ elements for some choice of prime $q$, 
where $q^v \cong 1 \text{ mod } p$, and $q^i \not\cong 1 \text{ mod } p$ for any $i < v$ (in general
we will notate this idea that $v$ is the smallest value such that $q^v \cong 1 \text{ mod } p$ 
by saying that $O(q \text{ mod } p) = v$).  We then take the homomorphism $h$ to be any homomorphism from 
$G$ to $R^{\times}$ whose kernel is of index $p$.  

We discuss these groups here because our characterization of the supersolvable numbers 
comes from a larger article by Gerhard Pazderski (\cite{pazderski}) which also derives the above characterization of ordered Sylow numbers and a characterization of metacyclic numbers, and Pazderski
frequently used these R'edei skew products as counterexamples, including once in the supersolvable case. 

We now give some properties of these groups, and the interested reader may refer to the article
for formal proofs.
In his article, R'edei proves that this group
 $RG$ is uniquely determined by $p,q,$ and $u$, and is independent of the 
choice of homomorphism with kernel of index $p$.  He further demonstrates two elements that are noncommutative, and
then proves that all proper subgroups of $RG$ are abelian by showing that any two noncommutative elements
generate the whole group.

Groups of this kind also have a nice presentation.  One can define the group $R(p,q,u)$ to be generated by elements
$a,b_0,\dots,b_{v-1}$ such that 
$$a^{p^u} = b_0^q = \dots = b_{v-1} = e,\ b_ib_j = b_jb_i (i \neq j),$$
$$a^{-1}b_ia = b_{i+1}, 0\leq i \leq v-1,$$
$$a^{-1}b_{v-1}a = b_0^{c_0}\dots b_{v-1}^{c_{v-1}},$$
with the $c_i$ the coefficients of an irreducible factor $x^v - c_{v-1}x^{v-1} - \dots - c_1x - c_0$ of
$\frac{x^p-1}{x-1} \text{ mod } q$, part of the $p$th cyclotomic polynomial mod $q$.  We give this example to
illustrate how Pazderski constructs one of these
groups as well as other
groups with similar presentations as part of his determination of the supersolvable numbers.

\subsection{Preliminaries of Supersolvable Groups}

\begin{definition}

A group $G$ is called \emph{supersolvable} if there exists a normal series
$$\{e\} = R_0 \triangleleft R_1 \triangleleft \dots \triangleleft R_r = G$$ 
such that $R_i \triangleleft G$ for all $0 \leq i \leq r$, and $R_{i+1} / R_i$ is
cyclic for all $0 \leq i \leq r-1$.  

\end{definition}

Supersolvable groups arise naturally in number theory: if $E/F$ is a finite Galois extension of a 
$p$-adic field $F$, then the Galois group of $E/F$ is supersolvable.  

Subgroups and quotient groups of supersolvable groups are supersolvable, and finite
nilpotent groups are supersolvable.  Further, finite supersolvable groups are solvable, 
so for finite groups, supersolvable is a group property between nilpotent and
solvable in strength.  We prove one particular property of supersolvable groups that
we will make use of later:

\begin{lemma} \label{lem:supord}

Every finite supersolvable group has an ordered Sylow tower.

\end{lemma}

This proof is adapted from \cite[p.158]{hall}.

\begin{proof}

First, we demonstrate that if we can choose a normal series for $G$ with successive quotient groups cyclic,
we can in fact choose a normal series in which the successive quotient groups are cyclic of prime order.
Let us suppose that for some $R_i$ in the normal series, $R_{i-1}/R_i$ is cyclic with order 
$p_1p_2\dots p_k$, where $p_1,p_2,\dots,p_k$ are (not necessarily distinct) primes.  Then $R_{i-1}/R_i$ 
has unique, hence characteristic cyclic subgroups $P_1,P_2,\dots,P_{k-1}$ of order 
$p_1,p_1p_2,\dots,p_1p_2\dots p_{k-1}$.  In particular, since $R_{i-1} \triangleleft G$, 
each $P_i$ is normal in $G/R_i$.  The liftings $P_i^*$ by the Correspondence Theorem satisfy 
$R_{i-1} \triangleright P_{k-1}^* \triangleright \dots 
\triangleright P_1^* \triangleright R_i$, each is normal in $G$, and
their successive quotients are cyclic of prime order.  Using this strategy, we can refine every quotient in 
the supersolvable group's normal series, creating a new normal series where every quotient is cyclic of
prime order.

Now we know that we may find a normal series in $G$ of the form 
$G = M_0 \triangleleft M_1 \triangleleft \dots \triangleleft M_n = \{e\}$ such that
$M_{i-1}/M_i$ is of prime order $m_i$.  We prove that we can always choose this series to also 
satisfy $m_{i-1} \geq m_i$ for all $1 \leq i \leq n$.  Suppose we have some index $k$ such that
$M_{k-1}/M_k$ is of prime order $p$, $M_k/M_{k+1}$ is of prime order $q$, and $p < q$.  Then
$M_{k-1}/M_{k+1}$ is of order $pq$ and has a characteristic subgroup of order 
$q$.  This corresponds to a normal subgroup $M_k^* \triangleleft M_{k-1}$ such that $M_{k-1}/M_k^*$ is of
order $q$ and $M_k^*/M_{k-1}$ is of order $p$.  We can repeatedly apply
this whenever $m_{i-1} < m_i$ to obtain a normal series for $G$ in which consecutive quotients have
nonincreasing order.

Finally, if we have a block $R_i \triangleleft R_{i+1} \triangleleft \dots \triangleleft R_{i+k}$ in the normal series
with consecutive quotients of the same prime order $m$, then we may simply remove the intermediate pieces to create a 
single piece $R_i \triangleleft R_{i+k}$ with quotient group of order $m^k$, which is the largest power of $m$ dividing
$|G|$ by the above construction.  Condensing the series in this manner 
for all distinct prime divisors of $|G|$ we obtain an ordered Sylow tower for $G$ as desired.       

\end{proof}

As an immediate consequence of the above proof, every supersolvable group $G$ has a normal subgroup of order equal to
the largest prime factor dividing $|G|$. 

\subsection{Describing the Criteria for Supersolvable Numbers}

We will devote this section to gaining a thorough understanding of the conditions for
a positive integer $n$ to be a supersolvable number, as given in \cite[p.335]{pazderski}.  Recall that $\psi$ is
the multiplicative function defined on prime powers as
$\psi(p^k) = (p^k-1)(p^{k-1}-1)\dots(p-1)$.

\begin{theorem}

Let $n$ be a positive integer with standard prime factorization $p_1^{a_1}p_2^{a_2} \dots p_r^{a_r}$, where 
$p_1 < p_2 < \dots < p_r$.  Then every group of order $n$ is supersolvable if and only if:

(1) For all $1 \leq i \leq r$, the distinct prime factors of $(n,\psi(p_i^{a_i}))$ are the
same as those of $(n,p_i-1)$.

(2) If there exists $i \neq k$ such that $p_i \leq a_k$ (i.e. the value of some prime factor of $n$ is
less than the multiplicity of another), then

	(a) There does not exist a prime $p_j$ such that $p_i | p_j-1$ and $p_j | p_k-1$, and
	
	(b) $a_i \leq 2$, and if $a_i = 2$, then $p_i^2 | p_k - 1$.

\end{theorem}

The set of conditions will collectively be referred to as (*) for the remainder of this
section.  Unlike in previous sections, this set of conditions is rather disjointed and 
not easy to make sense of.

Let us determine how a positive integer $n$ would violate (*).  We first consider condition (1).  
By definition $(p_i-1) | \psi(p_i^{a_i})$, so every prime factor of $(n, p_i-1)$ is also a prime factor of
$(n,\psi(p_i^{a_i}))$.  To violate (1), there must be some prime $p_j | n$ such that $p_j | \psi(p_i^{a_i})$
but $p_j \nmid p_i - 1$.  Using the definition of $\psi$, we may conclude that $p_j | p_i^k - 1$ for some
$2 \leq k \leq a_i$.  Thus, $n$ violates condition (1) if and only if it has a factor of the form $pq^v, v \geq 2$, where 
$p | q^v - 1$ and $p \nmid q^i-1$ for $i < v$.

We consider condition (2).  For this condition to be relevant, there must be a prime $p$ dividing $n$ such that
$q^p$ also divides $n$ for some prime $q$.  Obviously, condition (2)(a) is violated if and only if
there exists a prime $r$ dividing $n$ such that $p | r-1$ and $r | q-1$.  
Condition (2)(b) is violated if and only if $p^3 | n$, or if $p^2 | n$ but $p^2 \nmid q-1$.

In total, that are four kinds of factors $n$ could have that would violate (*): one by condition (1),
one by condition (2)(a), and two by condition (2)(b).  These will be summarized as part of the verification
of the validity of (*) in the following section.

\subsection{Verifying the Criteria for Supersolvable Numbers}   

\begin{theorem}

A positive integer $n$ is a supersolvable number if and only if it satisfies (*).

\end{theorem}

This proof will make use of representations of groups as matrices over a finite field.  
We will require the following lemma from \cite[p.365]{scott}:

\begin{lemma} \label{lem:monom}

If $G$ is a finite supersolvable group, then every representation of $G$ over an algebraically closed field with characteristic relatively prime to $|G|$ is equivalent
to a monomial representation (i.e. a representation such that every
element of $G$ is mapped to a matrix with exactly one nonzero entry in each row and column). 

\end{lemma}

\begin{proof}

We proceed by induction on the order of $G$ (note that the base case is trivial).  Since subgroups and
quotient groups of supersolvable groups are supersolvable, our inductive hypothesis will imply that
every representation of a subgroup or quotient group of $G$ is equivalent to a monomial representation.

It suffices to show the lemma is true for irreducible representations.  Let $T$ be an irreducible representation
of $G$ acting on a vector space $V$.  If $T$ is not faithful (aka one-to-one as a mapping), then
the representation of $T$ over $G$ is isomorphic to a representation of $T$ over $G/Ker(T)$, which is monomial
by the inductive hypothesis.  We may therefore assume $T$ is faithful.  Furthermore, if $G$  abelian, then
any irreducible representation of $G$ is monomial, so we may assume that $G$ is not abelian.

Suppose $Z(G)$ were a maximal abelian normal subgroup of $G$.  By assumption $G/Z(G)$ is supersolvable, so
has a normal series with prime steps.  Since $G$ is not abelian, $G/Z(G)$ is not cyclic, so there is a proper subgroup
$H$ of $G$ such that the normal series for $G/Z(G)$ terminates 
with the relation $\{e\} \triangleleft H/Z(G)$, where $H \triangleleft G$, and
$H/Z(G)$ must have prime order.  But then certainly $H$ is abelian, which contradicts the maximality of $Z(G)$
as an abelian normal subgroup of $G$.     

It follows that there exists an abelian normal subgroup $H$ of $G$ such that $Z(G) \subsetneq H$.  
We decompose the space $V$ into $V_1 \oplus V_2 \oplus \dots \oplus V_r$ such that $H$ 
acts as a group of scalars on each $V_i$, and
each $V_i$ is maximal with this property.   If $r = 1$, then $H$ acts as a group of scalars on all of $V$, so that
$H \subseteq Z(G)$, a contradiction.  Therefore we may assume that $r \geq 2$.  

Let $g \in G$, $h \in H$, and $v \in V_i$ for some $i$.  Then letting $h' = g^{-1}hg$ by $H \triangleleft G$,
we have $h(gv) = gg^{-1}h(gv) = g(g^{-1}hg)v = g(h')v = g(cv)$, where $c$ is a scalar produced by the action
of $h'$.  Note that $c$ depends on our initial choices of $g$ and $h$, but not on the choice of $v \in V_i$.
It follows that $gV_i \subset V_j$ for some $j$.  Since by construction $V = \oplus_{i=1}^r V_i$ and
$gV = V$, it must be that $G$ induces a permutation of the $V_i$.  Since $V$ is assumed to be irreducible, this
permutation must be a transitive action of $G$ on the $V_i$.

Let $Q$ be the subgroup of $G$ fixing $V_1$.  It must be that $|G : Q| = r$, so $G$ is equivalent to the disjoint union
of cosets of the form $Qu_i$, where $V_1u_i = V_i$.  By assumption $r > 1$ so that $Q$ is a proper subgroup of $G$,
so the representation of $Q$ over $V_1$ is monomial.  We let $B_1$ denote a basis of $V_1$ over which 
$Q$ is represented by monomial matrices.  Certainly the disjoint union of the sets
$B_1u_i$, $1 \leq i \leq r$ forms a basis for $V$.  Now, if $b \in B_1$ and $g \in G$,
$g(bu_i) = (gu_i)b = (qu_j)b = u_j(qb) = u_j(cb') = c(b'u_j)$, where $c$ is a scalar and $b' \in B_1$.  
Thus, $B$ is a monomial basis for $G$ over $V$.

\end{proof}

Now for the verification of supersolvable numbers (recall that $O(q \text{ mod } p) = v$ means that
$q^v \cong 1 \text{ mod } p$ but $q^i \not\cong 1 \text{ mod } p$ for $0 < i < v$):

\begin{proof}

We prove by induction that every $n$ satisfying (*) is a supersolvable number by showing that all groups of order $n$ have a 
normal subgroup whose order is equal to $p_r$, the largest prime dividing $n$ (recall that as a result of
\ref{lem:supord} the converse is true, that supersolvable groups have such a normal subgroup).  The base case is trivial.
Suppose that for a given $n$, all numbers less than $n$ satisfying 
(*) are supersolvable numbers with a normal subgroup of order their largest prime factor, and let $G$ be an
arbitrary group of order $n$.  If $n$ satisfies (*), all of its divisors must also satisfy (*), so all proper
subgroups and quotient groups of $G$ are supersolvable and have a normal subgroup of largest prime order.  In the
case that $n = p^a$ is a prime power, $n$ clearly satisfies (*), and since $G$ is a $p$-group, it must be finite nilpotent,
and therefore supersolvable.  Furthermore, since $p$-groups have nontrivial center, 
there exists a subgroup of order $p$ contained
in $Z(G)$, and this subgroup must then be normal in $G$.

Otherwise, let $n = p_1^{a_1} \dots p_r^{a_r}$, where $r \geq 2$ and $p_1 < p_2 < \dots < p_r$.  It suffices 
to show that $G$ has a normal subgroup $N$ of order $p_r$.  If it does, then by the inductive hypothesis
$G/N$ is supersolvable, so has a normal series 
$\{e\} = N/N \triangleleft R_1/N \triangleleft \dots \triangleleft G/N$ with consecutive quotients cyclic.  
This corresponds to a partial normal series $N \triangleleft R_1 \triangleleft \dots \triangleleft G$ for $G$, where
consecutive quotients are cyclic.  But since $N$ is of prime order, this trivially extends to a complete normal series
for $G$ with consecutive quotients cyclic, so then $G$ is supersolvable.  

We claim that $(p_i^{a_i} \dots p_r^{a_r},\psi(p_i^{a_i})) = 1$ for all $1 \leq i \leq r$.  If not, then 
there exists $i < j$ such that $p_j | \psi(p_i^{a_i})$.  By (1) in (*), it must also be that $p_j | p_i - 1$, but
this is absurd, since $p_j > p_i$.  This implies that $G$ has an ordered Sylow tower, so it contains a system of 
Sylow subgroups $\{P_1,\dots,P_r\}$ such that $P_i \triangleleft P_1 \dots P_i$, and in particular 
$P_r \triangleleft G$.  Let $B = P_1 \dots P_{r-1}$.  Then clearly $B \cap P_r = \{e\}$, so
$G = BP_r$.

Suppose that $P_r$ is not a minimal normal subgroup of $G$, that is, there exists $M \triangleleft G$
such that $\{e\} \subsetneq M \subsetneq P_r$.  We prove that $M \cap Z(P_r)$ is nontrivial.  
Since $M \triangleleft P_r$ necessarily, $M$ is a union of conjugacy classes of $P_r$.  
All of these conjugacy classes have $p_r$-power order except for those 
consisting of single elements in $Z(P_r)$.  If the only element of $M$ from $Z(P_r)$ were
$\{e\}$, we would have $|M| \cong 1 \text{ mod } p_r$.  Since the only subgroup of $P_r$ not of 
$p_r$-power order is $\{e\}$, we would have $M = \{e\}$, but this contradicts the definition of $M$.  It follows
that $|M \cap Z(P_r)| > 1$.  Furthermore, $Z(P_r)$ is characteristic in $P_r$ and $P_r \triangleleft G$, so
$Z(P_r) \triangleleft G$.  Since also $M \triangleleft G$, we must have that $M \cap Z(P_r)$ is a nontrivial
normal subgroup of $G$ that is certainly a proper subgroup of $P_r$.  
Altogether this means that from now on we may assume that $M \subseteq Z(P_r)$; 
if not, we may replace $M$ by $M \cap Z(P_r)$, which
is also normal in $G$ and a nontrivial proper subgroup of $P_r$, but is contained in $Z(P_r)$. 

Now, $BM$ is a proper subgroup of $G$, so since $p_r$ divides $|BM|$, we may apply the inductive hypothesis
to conclude that $BM$ has a normal subgroup $N$ of order $p_r$, necessarily contained in $M$.  
Since $N \subseteq M \subseteq Z(P_r)$, clearly $N \triangleleft P_r$.  Since also $N \triangleleft B$, it 
must be that $N \triangleleft BP_r = G$, and we have found a normal subgroup in $G$ of order $p_r$.

From now on we may assume that $P_r$ is a minimal normal subgroup of $G$.  This means that $P_r$ is elementary
abelian by Lemma \ref{lem:elem}, and thus isomorphic to $C_{p_r}^{a_r}$, so we may think of its elements
as vectors in $\mathbb{F}_{p_r}^{a_r}$.  Since $P_r \triangleleft G$, conjugating the elements
of $P_r$ by an element $b \in B$ induces an automorphism on $P_r \cong \mathbb{F}_{p_r}^{a_r}$, which may be described 
as matrix multiplication by associating to each $b \in B$ a matrix  $M_b \in GL_{a_r}(\mathbb{F}_{p_r})$.  Given this 
representation, for $b \in B$ and $v \in P_r$, we associate $bvb^{-1}$ with $M_bv$.  We will notate
this representation as $\Delta: B \rightarrow GL_{a_r}(\mathbb{F}_{p_r})$. It is 
irreducible, in the sense that no proper subgroup of $P_r$ other than $\{e\}$ forms a vector subspace that is
closed under the action induced by elements of $B$.  
If there were such a subgroup $M$, then $M$ closed under conjugation
by $B$ implies $M \triangleleft B$, and since $P_r$ is abelian we have $M \triangleleft P_r$, but then
$M \triangleleft BP_r = G$, contradicting that $P_r$ is minimal normal in $G$.

Suppose that this representation of $B$ as matrices is not one-to-one (aka faithful).  Then in terms of the
group structure, there is some nontrivial element of $B$ that lies in the kernel of the conjugation action
of $B$ on $P_r$.  Let $Z(G, P_r)$ denote the centralizer of $P_r$ in $G$.  Then this kernel is $B \cap Z(G, P_r)$, and it must be nontrivial.  It is normal in $B$ as the kernel of 
a homomorphism, and it is clearly normalized by $P_r$, so it must be normal in $BP_r = G$.  Then we may apply our
inductive hypothesis to $G/(B \cap Z(G, P_r))$ to find a normal subgroup $R/(B \cap Z(G, P_r))$ of order
$p_r$.  This corresponds to a normal subgroup $R \triangleleft G$ whose order is divisible by $p_r$, but is not
divisible by $p_r^2$, since $|B \cap Z(G, P_r)|$ is relatively prime to $p_r$.  Then 
$R \cap P_r$ is normal in $G$ since $R$ and $P_r$ are, and it must have order exactly $p_r$.

We may now assume that our representation $\Delta$ of $B$ is one-to-one (or faithful).  Then 
$P_i, 1 \leq i \leq r-1$ is a subgroup of $B$ represented faithfully as a subgroup of 
$GL_{a_r}(\mathbb{F}_{p_r})$, so that $|P_i|$ divides $|GL_{a_r}(\mathbb{F}_{p_r})|$, or
$p_i^{a_i} | \psi(p_r^{a_r})$.  Because of our condition (*), 
we must have that since $p_i | (n,\psi(p_r^{a_r}))$, we also have $p_i | p_r - 1$ for
$1 \leq i \leq r-1$.

We now show that if there exists $p_i \leq a_r ~ (i \neq r)$, then the elements of $P_i$ are mapped to multiples
of the identity matrix by $\Delta$.  First, note that by (*), we must have $a_i \leq 2$, so that 
$P_i$ is abelian.  Also, we claim that elements of $P_i$ commute with elements of $P_j$ for 
$i \neq j$, $1 \leq i,j \leq r-1$.  If not, then $P_iP_j$ cannot be a nilpotent group, so by the
criterion for nilpotent numbers we conclude that $(p_i^{a_i}p_j^{a_j}, \psi(p_i^{a_i}p_j^{a_j})) > 1$, so is 
divisible by either $p_i$ or $p_j$.  However, this subgroup satisfies (*), so by condition (1) it must be that
either $p_i | p_j - 1$ or $p_j | p_i - 1$.  Certainly we cannot have $p_i | p_j - 1$, because also
$p_j | p_r - 1$, contradicting (2)(a) of (*).  However, if $p_j | p_i - 1$, then certainly
$p_j < p_i \leq a_r$, but then since $p_i | p_r - 1$ we have again contradicted (2)(a).  This proves that
the elements of $P_i$ must commute with all elements in $B$ (since $P_i$ is also abelian).  Also, 
$p_i | p_r - 1$, and by (*) condition (2)(b) $a_i \leq 2$ and if $a_i = 2$ then $p_i^2 | p_r - 1$, so in any case
$p_i^{a_i} | p_r - 1$, so the $p_i^{a_i}$th roots of unity are contained in $\mathbb{F}_{p_r}$.  Thus, we may decompose
the representation of $P_i$ into irreducible subrepresentations, but these each have dimension 1 since $P_i$ is abelian,
and we may conclude that $\Delta$ maps the elements of $P_i$ to scalar matrices.

Now, since $B$ is supersolvable, we may apply Lemma \ref{lem:monom} to conclude that 
$\Delta$ is equivalent to a monomial representation $\Delta'$ over a field extension of $\mathbb{F}_{p_r}$.  
In $\Delta'$, we look at the set of diagonal matrices $D$ as a subgroup of $B$. 
It is clear that
$D$ is abelian, and it is easy to check that $D \triangleleft B$.  Since $B$ is monomial, all of its elements
are matrices where the rows are some permutation of a diagonal matrix, so the group $B/D$ is equivalent to
a subgroup of the permutation group on $a_r$ elements.  We look at the orbit of an arbitrary element of 
$B/D$, and let $d$ denote the size of the orbit.  Certainly $d \leq a_r$ and $d$ divides $|B : D|$.  
Now, any prime dividing $|B : D|$ must be one of the $p_i$.  But if $p_i \leq a_r$, we already know that
the elements of $B$ contained in $P_i$ are mapped to the set of diagonal matrices $D$, so cannot contribute
to a nontrivial orbit in $B/D$.  Thus there are no $p_i$ that can divide $d$, so $d = 1$, and
$B = D$.  But this means that $B$ is finite abelian, so it is mapped to a direct sum of cyclic groups under 
$\Delta$.  Since $\Delta$ is irreducible, $B$ must be mapped to just one cyclic group, so since $\Delta$ is also
faithful, $B$ itself must be cyclic.  

Since now $B$ is a cyclic group acting irreducibly on $\mathbb{F}_{p_r}^{a_r}$, 
it must be that $a_r = 1$.  But then our
subgroup $P_r \triangleleft G$ actually has order $p_r$, and we are done with this direction.

Now suppose that $n$ does not satisfy (*).  Then as noted in the previous section, $n$ must have one of the following
factors (where $p$,$p'$,$q$ are primes):
$$f_1: pq^v, v \geq 2, v = O(q \text{ mod } p)$$
$$f_2: pp'q^p, p|p'-1, p'|q-1$$
$$f_3: p^2q^p, p^2 \nmid q-1$$
$$f_4: p^3q^p$$
For factors $f_2,f_3,f_4$, note that 
by Fermat's Little Theorem certainly $p | q^{p-1}-1$, 
so in particular $p | \psi(q^p)$, which by condition (1) of (*) implies
that $p | q-1$.  Thus, we may add the assumption that $p | q-1$ to $f_2,f_3,$ and $f_4$.
  
For each of these possible $f_i$, we explicitly construct a nonsupersolvable group $F_i$ of order $f_i$.  Then we may
construct the group $F_i \times C_{n/f_i}$ that contains a nonsupersolvable subgroup, so cannot itself
be supersolvable.

Case 1: We construct the group $R(p,q,1)$ defined in ~\ref{sec:redei}.  If this group is 
supersolvable, there exists a normal series with consecutive quotients of prime order, and in
particular the series would need to terminate with a normal subgroup of $R(p,q,1)$ of prime order.  
We thus may prove that $R(p,q,1)$ is not supersolvable by demonstrating that it
has neither a normal subgroup of order $p$ nor a normal subgroup of order $q$.

Let $G = \{e,\alpha,\alpha^2, \dots, \alpha^{p-1}\}$, and let $r$ represent a generic element of $R$, 
the finite field of $q^v$ elements.  Let $h: G \rightarrow R^{\times}$ be the skew
homomorphism.  Note that this group $RG$ has order $pq^v$, 
so any group of order $p$ is actually a Sylow $p$-subgroup.  We can construct at least one subgroup
of order $p$, namely the subgroup generated by $(0,\alpha)$.  We try conjugating a general element:
$$(a,\alpha^j)(0,\alpha^i)(-h(\alpha^{-j})a, \alpha^{-j}) = (a,\alpha^{i+j})(-h(\alpha^{-j})a, \alpha^{-j}) = 
(a-h(\alpha^{i})a, \alpha^{i}).$$
This lies back in the original subgroup only if $h(\alpha^i) = 1$.  However, $h$ is not the trivial homomorphism
by construction, so this conjugation action cannot preserve the subgroup generated by $(0,\alpha)$.  There cannot
be a different normal subgroup of order $p$, because then as a normal Sylow $p$-subgroup it would be unique in
$RG$, which it obviously isn't.  Thus $RG$ has no normal subgroup of order $p$.

Now, it is easy to see that a subgroup of order $q$ cannot contain an element $(a,\alpha)$ 
with nontrivial second coordinate, since then $(a,\alpha)^q$ could not be the identity, as $\alpha^q \neq e$.  
Thus, any subgroup of order $q$ is a subgroup generated by an element of the form $(a,e)$.  Suppose such a subgroup were 
normal in $RG$.  Then we conjugate: 
$$(r,\alpha)(a,e)(-h(\alpha^{-1})r, \alpha^{-1}) = (r+h(\alpha)a, \alpha)(-h(\alpha^{-1})r, \alpha^{-1})
= (h(\alpha)a,e).$$
Now, this subgroup already contains the $q$ elements \\ 
$(0,e),(a,e),(2a,e),\dots,((q-1)a,e)$, so it should not contain
any more.  Thus, $h(\alpha)a = ca$ for some $2 \leq c \leq q-1$ (note that $c \neq 0$ since $0$ is not a unit,
and $c \neq 1$ because $h$ is nontrivial).  
After repeated conjugation by $(r,\alpha)$, $p$ times in all, we may see that
$a = c^pa$, so $c^p \cong 1 \text{ mod } q$.  Now, by Fermat's Little Theorem,
$c^{q-1} \cong 1 \text{ mod } q$, so the order of $c$ divides $q-1$.  However, by assumption $p \nmid q-1$ (recall that
$v \geq 2$ is the smallest value such that $q^v \cong 1 \text{ mod } p$), so then the order of $c$ is not
$p$, and $c^p \not\cong 1 \text{ mod } q$, a contradiction.  It follows that $RG$ has no normal subgroups of order
$q$ either, so it cannot be supersolvable.  

Case 2: We define a group of order $pp'q^p$ by a presentation similar to those for R'edei's first-order non-abelian groups.
Let $\rho$ be a number such that $O(\rho \text{ mod } p') = p$, and let $\sigma$ be a number such that
$O(\sigma \text{ mod } q) = p'$ (note that these exist because $p | p' - 1$ and $p' | q-1$).  
Let $a,a',b_1,b_2,\dots,b_q$ be distinct elements of a group $H$ such that
$$a^p = a'^{p'} = b_1^q = b_2^q = \dots = b_p^q = e,\ b_ib_j = b_jb_i,$$
$$a^{-1}a'a = a'^{\rho},$$
$$a^{-1}b_ia = b_{i+1} (1 \leq i \leq p-1), a^{-1}b_pa = b_1,$$
$$a'^{-1}b_1a' = b_1^{\sigma}.$$

It is clear that $b_i = a^{1-i}b_1a^{i-1}, 1 \leq i \leq p$.  Furthermore,
$$a'^{-1}b_ia' = a'^{-1}a^{1-i}b_1a^{i-1}a' = $$
$$a^{1-i}a^{i-1}a'^{-1}a^{1-i}b_1a^{i-1}a'a^{1-i}a^{i-1} = 
a^{1-i}a'^{\rho^{i-1}}b_1a'^{\rho^{1-i}}a^{i-1} = $$
$$a^{1-i}b_1^{\sigma^{\rho^{1-i}}}a^{i-1} = b_i^{\sigma^{\rho^{1-i}}}.$$

Clearly $q$ is the largest prime factor of $|H|$, so if $H$ has no normal subgroup of order $q$, it has no 
ordered Sylow tower, and thus is not supersolvable.  Suppose the contrary, that there is some normal subgroup
of $H$ of order $q$.  Then it is generated by an element $b$ of order $q$, which must have the form
$b =  b_1^{x_1}\dots b_p^{x_p}$, where not all of the $x_i$ are $0 \text{ mod } q$.  Now, 
$$a^{-1}ba = a^{-1}b_1^{x_1}aa^{-1}b_2^{x_2}a \dots a^{-1}b_px^pa = 
b_1^{x_p}b_2^{x_1}b_3^{x_2}\dots b_p^{x_{p-1}},$$
and in general $a^{-i}ba^i = b_1^{x_{1-i}}b_2^{x_{2-i}} \dots b_p^{x_{p-i}}$, where the subscripts of the
$x$s are mod $p$.  By normality,
$a^{-1}ba = b^c$ for some $1 \leq c \leq q-1$, so that also $a^{-i}ba^i = b^{c^i}$ for all $i$.
Comparing the exponents of $b_p$ for $1 \leq i \leq p$, we see that $x_i \cong x_pc^{p-i}$.  
By assumption at least one $x_i$ is not $0 \text{ mod } q$, so then $x_p$ is not $0 \text{ mod } q$.  But if
$x_p$ is not $0 \text{ mod } q$, then the above equation implies that in fact all of the $x_i$ are nonzero mod $q$.

Now, 
$$a'^{-1}ba' = a'^{-1}b_1^{x_1}a'a'^{-1}b_2^{x_2}a'\dots a'^{-1}b_p^{x_p}a' =
b_1^{x_1\sigma}b_2^{x_2\sigma^{\rho^{-1}}} \dots b_p^{x_p\sigma^{\rho^{1-p}}}.$$
By normality, $a'^{-1}ba' = b^d$ for some $1 \leq d \leq q-1$.  We compare the exponents of $b_1$, and see that
$x_1\sigma = x_1d$, so necessarily $\sigma = d$.  Comparing the exponents of $b_2$, we see that 
$x_2\sigma^{\rho^{-1}} = x_2d = x_2\sigma$, so then
$\sigma^{\rho^{-1}} \cong \sigma \text{ mod } q$, or $\sigma^{\rho^{-1}-1} \cong 1 \text{ mod } q$.  By assumption
$\sigma^{p'} \cong 1 \text{ mod } q$, so $\rho^{-1} \cong 1 \text{ mod } p'$, but this contradicts our assumption that
$O(\rho \text{ mod } p') = p$.  Thus, there is no normal subgroup of $H$ of order $q$, and this group is not supersolvable.

Case 3: We define a group of order $p^2q^p$ with a presentation.  Let $\rho$ be a number such that 
$O(\rho \text{ mod } q) = p$, and let $a,b_1,b_2,\dots,b_p$ be distinct elements of a group $H$.  We define $H$ by the relations
$$a^{p^2} = b_1^q = b_2^q = \dots = b_p^q = e;\ b_ib_j = b_jb_i,$$
$$a^{-1}b_ia = b_{i+1}, 1 \leq i \leq p-1,$$
$$a^{-1}b_pa = b_1^{\rho}.$$

As before, $q$ is the largest prime divisor of $H$, so we wish to show that $H$ contains no normal subgroup of order
$q$.  Suppose to the contrary that we have such a group.  It would be generated by an element
$b = b_1^{x_1}\dots b_p^{x_p}$, where not all of the $x_i$ are $0 \text{ mod } q$.  Note that
$a^{-1}ba = a^{-1}b_1^{x_1}a\dots a^{-1}b_p^{x_p}a = b_1^{\rho x_p}b_2^{x_1} \dots b_p^{x_{p-1}}$.  We must have that
$a^{-1}ba = b^c$ for some $1 \leq c \leq q-1$, and also $a^{-i}ba^i = b^{c^i}$.  Comparing the exponents of $b_p$ for 
$1 \leq i \leq p-1$ we find that $x_i \cong x_pc^{p-i} \text{ mod } q$.  As before, this allows to conclude that $x_p$ is not $0 \text{ mod } q$.  For $i = p$, we have that $\rho x_p \cong x_pc^p \text{ mod } q$, so since $x_p$ is nonzero, 
$\rho \cong c^p$ mod $q$.  But then $1 \cong \rho^p \cong c^{p^2}$, so clearly $O(c \text{ mod } q) = p^2$, and so
$p^2 | q-1$, a contradiction.  It follows that this normal subgroup does not exist, and this group is not supersolvable.  

Case 4: We wish to construct a group of order $p^3q^p$ that is not supersolvable.  We may assume that $p^2 | q-1$, since otherwise we could instead fall back to Case 3.  Let $\rho$ be a number such that $O(\rho \text{ mod } q) = p^2$.  
Let $a_1,a_2,b_1,\dots,b_p$ be distinct elements of a group $H$.  We define $H$ by the relations
$$a_1^p = a_2^{p^2} = b_1^q = \dots = b_p^q = e;\ b_ib_j = b_jb_i,$$
$$a_1^{-1}a_2a_1 = a_2^{1+p},$$
$$a_1^{-1}b_ia_1 = b_{i+1},\ a_1^{-1}b_pa_1 = b_1,$$
$$a_2^{-1}b_1a_2 = b_1^{\rho}.$$

It is clear that $b_i = a_1^{1-i}b_1a_1^{i-1}$.  Furthermore, we have that
$$a_2^{-1}b_ia_2 = a_2^{-1}a_1^{1-i}b_1a_1^{i-1}a_2 = $$
$$a_1^{1-i}a_1^{i-1}a_2^{-1}a_1^{1-i}b_1a_1^{i-1}a_2a_1^{1-i}a_1^{i-1} = 
a_1^{1-i}a_2^{-(1+p)^{(1-i)}}b_1a_2^{(1+p)^{(1-i)}}a_1^{i-1} = $$
$$a_1^{1-i}b_1^{\rho^{(1+p)^{(1-i)}}}a_1^{i-1} = b_i^{\rho^{(1+p)^{(1-i)}}}.$$  

We can make one further simplification here.  The value $(1+p)^{(1-i)}$ occurred as an exponent of $a_2$, so we are only concerned with its value mod $p^2$.  But for any positive $i$, $(1+p)^{(i-1)}$ has a binomial expansion in which every term is divisible by $p^2$ except for $1+(i-1)p$.  Then $(1+p)^{(1-i)}$ is the multiplicative inverse of this, but is not hard to
check that this is $1+(1-i)p$, since $(1+(i-1)p)(1-(i-1)p) = 1-(i-1)^2p^2 \cong 1 \text{ mod } p^2$.  Thus, we may simplify the above relation to $a_2^{-1}b_ia_2 = b_i^{\rho^{1+(1-i)p}}$.

As in previous cases, $q$ is the largest prime factor of $|H|$, so it suffices to show that $H$ has no normal subgroup
of order $q$.  We suppose there exists such a normal subgroup.  It is generated by some element 
$b = b_1^{x_1}\dots b_p^{x_p}$, where not all of the $x_i$ are $0$ mod $q$.  Using the same argument as in case 2, we may
conclude that this in fact means that none of the $x_i$ are $0$ mod $q$.  We also have that
$a_2^{-1}ba_2 = b_1^{x_1\rho}b_2^{x_2\rho^{1-p}}\dots b_p^{x_p\rho^{1+p-p^2}}$.  By normality, it must be that
$a_2^{-1}ba_2 = b^c$ for some $1 \leq c \leq q-1$.  Comparing the exponents of $b_1$, we have that
$x_1\rho \cong x_1c $ mod $q$, so since $x_1$ is nonzero mod $q$, $\rho \cong c$ mod $q$.  Comparing the exponents of
$b_2$, we have that $x_2\rho^{1-p} \cong x_2c$ mod $q$, so 
$\rho^{1-p} \cong \rho$ mod $q$, and $\rho^p \cong 1$ mod $q$, but this contradicts that
$O(\rho \text{ mod } q) = p^2$.  Thus, the normal subgroup does not exist, and this group is not supersolvable.

To summarize, if $n$ is a positive integer that does not satisfy (*), then it must have one of the four kinds of factors
corresponding to the cases above.  For each factor, we may explicitly construct a group that is not supersolvable, and then any group containing it cannot be supersolvable.  In particular, this means that we can certainly construct a group of order $n$ that is not supersolvable, so $n$ is not a supersolvable number, and this concludes the proof.
  
\end{proof}

\section{Conclusion}

Characterizations of $P$ numbers for various properties $P$ of groups greatly help to construct a complete picture of 
the finite groups.  Characterizations similar to those here exist for solvable numbers, but these
proofs are fairly complex.  For more information, see \cite{pakianathan}, 
which gives many sources for these kinds of characterizations.

\clearpage
\pagestyle{empty}

\end{document}